\documentclass{article}

\usepackage{amsmath,amssymb,amsthm,enumerate,mathrsfs}
\usepackage[numbers]{natbib}
\usepackage{booktabs}
\usepackage{adjustbox}
\usepackage{subfigure,multirow}
\usepackage{appendix}
\usepackage{authblk}
\usepackage{geometry}
\usepackage[ruled,linesnumbered]{algorithm2e}
\usepackage{algpseudocode}%
\usepackage{longtable}
\usepackage{caption}
\usepackage{lipsum}
\usepackage{titlesec}

\newcommand{\x}{\mathbf{x}}
\newcommand{\yy}{\mathbf{y}}
\newcommand{\sss}{\mathbf{s}}
\newcommand{\g}{\mathbf{g}}

\newcommand{\dd}{\mathbf{d}}

\newcommand{\T}{\text{T}}
\newcommand{\BB}{Barzilai-Borwein }
\newcommand{\BBB}{Barzilai and Borwein }

\theoremstyle{plain}
\newtheorem{theorem}{Theorem}
\newtheorem{lemma}[theorem]{Lemma}
\newtheorem{proposition}[theorem]{Proposition}

\theoremstyle{remark}
\newtheorem{remark}{Remark}

\title{\bf A Parameterized \BB Method via Interpolated Least Squares\thanks{The work was supported by the National Natural Science Foundation of China (Project No. 12371099).}}

\author{Xin Xu\thanks{Corresponding author: xustonexin@gmail.com}}
\affil{School of Mathematics, Southwestern University of Finance and Economics, Chengdu, China}
\date{}

\begin{document}
\maketitle

\begin{abstract}
	The \BB(BB) method is an effective gradient descent algorithm for solving unconstrained optimization problems. Based on the observation of two classical BB step sizes, by constructing an interpolated least squares model, we propose a novel class of BB step sizes, each of which still retains the quasi-Newton property, with the original two BB step sizes being their two extreme cases. We present the mathematical principle underlying the adaptive alternating BB (ABB) method. Based on this principle, we develop a class of effective adaptive interpolation parameters. For strictly convex quadratic optimization problems, we establish the $R$-linear convergence of this new gradient descent method by investigating the evolution pattern of the ratio of the absolute values of the gradient components. Numerical experiments are conducted to illustrate our findings.
	
	
\end{abstract}

\section{Introduction}
In this paper, we consider the unconstrained optimization problem
\begin{equation}\label{generalqua}
\min_{\x\in\mathbb{R}^{n}} f(\x),
\end{equation}
where $f:\mathbb{R}^n\longrightarrow \mathbb{R}$ is continuously differentiable. A minimizer is denoted by $\x_{*}$. The gradient descent method for solving \eqref{generalqua} is an iterative method of the form 
\begin{equation}\label{gradient method}
\x_{k+1}=\x_{k}-\alpha_{k}^{-1}\g_{k},
\end{equation}
where $\x_{k}$ is the $k$th approximation to $\x_{*}$, $\g_{k}:=\nabla f(\x_{k})$ is the gradient of $f$ at $\x_{k}$, and $\alpha_{k}^{-1}$ is the step size. The classical steepest descent (SD) method dates back to Cauchy \cite{Cauchy2009Methodegeneralepour} with $\alpha_{k}$ is taken as the exact minimizer of $f(\x_{k}-\alpha_{k}^{-1}\g_{k})$. Nevertheless, in practice, the SD method is not efficient and usually converges slowly due to zigzag behavior \cite{Akaike1959successivetransformationprobability}. In order to improve the performance of SD method, \BBB(BB)  \cite{Barzilai1988TwoPointStep} introduced two novel scalars $\alpha_{k}$ by the following two least squares models
\begin{equation}\label{LS}
\min_{\alpha\in\mathbb{R}} \|\alpha \sss_{k-1} - \yy_{k-1}\|_{2}^{2},\quad \min_{\alpha\in \mathbb{R}} \|\sss_{k-1} -  {\alpha}^{-1}\yy_{k-1}\|_{2}^{2},
\end{equation} 
where $\sss_{k-1}=\x_{k}-\x_{k-1}$ and $\yy_{k-1}=\g_{k}-\g_{k-1}$. The solutions to \eqref{LS} are 
\begin{equation}\label{BB steps}
	\alpha_{k}^{BB1}:=\frac{\sss_{k-1}^{\T}\yy_{k-1}}{\sss_{k-1}^{\T}\sss_{k-1}},\quad \alpha_{k}^{BB2}:=\frac{\yy_{k-1}^{\T}\yy_{k-1}}{\sss_{k-1}^{\T}\yy_{k-1}},
\end{equation}
respectively. Assume that  $\sss_{k-1}^{\T}\yy_{k-1}>0$, by the Cauchy-Schwarz inequality \cite{Golub2013MatrixComputations}, we know that $\alpha_{k}^{BB1}\le\alpha_{k}^{BB2}$ holds. We refer to the gradient descent methods with step size $(\alpha_{k}^{BB1})^{-1}$ and $(\alpha_{k}^{BB2})^{-1}$ as the BB1 and BB2 methods, respectively. In the least squares sense, BB method approximates the Hessian of $f(\x_{k})$ using $\alpha_{k}^{BB1}I$ or $\alpha_{k}^{BB2}I$, where $I$ is the identity matrix.

Such simple and efficient step sizes yield an excellent $R$-superlinear convergence in the two dimensional strictly convex quadratic problems \cite{Barzilai1988TwoPointStep}, in which the authors showed that in any sequence of four consecutive iterations, at least one step exhibits superlinear convergence. For any dimensional strictly convex quadratic problems, BB method is shown to be globally convergent \cite{Raydan1993BarzilaiBorweinchoice} and the convergence is $R$-linear \cite{Dai2002Rlinearconvergence}. The superior performance of the BB method demonstrates that the behavior of the SD method cannot be attributed solely to its search direction, while the design of efficient step size is the critical factor.   

Although the BB method does not decrease the objective function value monotonically, extensive experiments show that it performs much better than the SD method \cite{Friedlander1999GradientMethodRetards,Fletcher2005BarzilaiBorweinMethod,Dai2002ModifiedTwoPoint,Yuan2018StepSizesGradient}. There are several intensive researches about the BB method. For example, \text{Zhou} et al. \cite{Zhou2006GradientMethodsAdaptive} proposed an adaptive alternating BB (ABB) method. \text{Dai} et al. \cite{Dai2006cyclicBarzilaiBorwein} developed a cyclic BB method, in which the same BB step size is reused for $m$ consecutive iterations. \text{Burdakov} et al. \cite{Burdakov2019StabilizedBarzilaiBorwein} presented a stabilized BB method. Huang et al. \cite{Huang2021EquippingBarzilaiBorwein, Huang2022accelerationBarzilaiBorweinmethod} considered to accelerate BB method by requiring quadratic termination for minimizing the two dimensional strongly convex quadratic function. \text{Ferrandi} et al. \cite{Ferrandi2023harmonicframeworkstepsize} proposed a harmonic framework for step size selection. Recently, Li et al. \cite{Li2024familyBarzilaiBorwein} presented a family of BB step sizes from the view point of scaled total least squares.  
For the \text{nonquadratic} problems, Raydan \cite{Raydan1997BarzilaiBorweinGradienta} proposed a globalized BB gradient method with the non-monotonic line search technique proposed by Grippo et al.  \cite{Grippo1991classnonmonotonestabilization} for solving large scale optimization problems. This non-monotonic line search technique not only aligns with the convergence characteristics of the BB method but also guarantees its convergence for general objective functions. Consequently, following this advancement, the BB method has been extensively applied to various optimization problems. To mention a few of them, Dai et al.  \cite{Dai2005ProjectedBarzilaiBorwein,Dai2005Newalgorithmssingly} developed projected BB-like methods for special quadratic programs arising from training support vector machine, that has a singly linear constraint in addition to box constraints. The BB method enjoys many important applications, such as constrained
image \text{deblurring} \cite{Bonettini2009scaledgradientprojection}, \text{nonlinear} least squares \cite{Mohammad2019StructuredTwoPoint}, variational inequality \cite{Dong2018Predictioncorrectionmethod}, \text{Stiefel} manifold optimization \cite{Gao2018NewFirstOrder}, deep learning \cite{Liang2019BarzilaiBorweinbasedadaptive}, spherical $t$-designs \cite{An2020Numericalconstructionspherical}, etc.  

As is well know, an algorithm that alternately uses two BB step sizes is usually better than an algorithm that uses only one BB step size. The reason is that alternating BB step size is beneficial to improving the algorithm's ability to scan the Hessian spectrum. Inspired by this, in order to fully exploit the spectral properties of the two BB step sizes, we attempt to use a unified form to handle these two step sizes. By observing the two original least squares models in  \eqref{LS}, a feasible approach is to unify the two least squares models into one least squares model by modifying the degree of $\alpha$, which provides us with a novel perspective on BB method. 

This paper makes following three contributions.
\begin{itemize}
	\item We introduce an interpolated least squares model in which the degree of $\alpha$ is a variable located in $[0,1]$. Thus, we obtain a class of BB-like step sizes with quasi-Newton properties, and the original two BB step sizes correspond to its two extreme cases.
	\item From the perspective of gradient element elimination, we provide an analysis of the ABB method. Based on this principle, we propose an effective interpolation parameter scheme. For strictly convex quadratic optimization problems, a remarkable result is that when the largest eigenvalue $\lambda$ of the Hessian is a large real number, $\alpha_{k}^{BB2}$ that satisfies the alternating criterion approximates $\lambda$, while $\alpha_{k}^{BB1}$ approaches the smallest eigenvalue, thereby achieving rapid elimination of gradient components. It is evident that the alternating strategy enhances the algorithm's ability to traverse the Hessian spectrum.
	\item For strictly convex quadratic optimization problems, we establish the $R$-linear convergence of this new gradient descent method by studying the variation pattern of the ratio of the absolute values of the gradient components. The convergence of many effective spectral gradient descent methods can be analyzed using this framework.
	
\end{itemize}

The remaining part of this paper is organized as follows. In Section \ref{sec:2}, we propose  parameterized \BB (PBB) method through an interpolated least squares model. In Section \ref{sec:convergence}, we establish the $R$-linear convergence of the PBB method. In Section \ref{sec:ABB}, we analyze the principle of the ABB method, which provides a basis for  selecting interpolation parameters in subsequent section. In Section \ref{sec:mk}, based on the analysis in the preceding section, we propose an effective interpolation parameter selection scheme. In Section \ref{sec:nonqua}, we generalize the PBB method to nonquadratic optimization problems. In Section \ref{sec:numerical}, we present the numerical results of the PBB method for both quadratic and nonquadratic optimization problems. In the final section, we present the conclusion and discuss future research directions. Throughout this paper, we refer to $\mathbf{1}$ and $\mathbf{0}$ as vectors of all ones and zeros, respectively, $\kappa(A)$ as the condition number of the matrix $A$. 

\section{Parameterized \BB (PBB) method}\label{sec:2}
In this section, we introduce an interpolated least squares model, which is a completeness of the least squares models in the original BB method. Based on this interpolated least squares model, we obtain a family of BB-like step sizes, and the original BB step sizes are its two extreme cases. We only consider the case $\sss_{k-1}^{\T}\yy_{k-1}>0$ in this section.

\subsection{An interpolated least squares model}
Observing the least squares models \eqref{LS} in the BB method, we find that the two models \eqref{LS} can be rewritten as follows
\begin{equation*}\label{LLS}
\min_{\alpha\in\mathbb{R}} \|\alpha^{1} \sss_{k-1} -\alpha^{0} \yy_{k-1}\|_{2}^{2},\quad \min_{\alpha\in \mathbb{R}} \|\alpha^{0}\sss_{k-1} -  \alpha^{-1}\yy_{k-1}\|_{2}^{2},
\end{equation*} 
that is, the degree of the variable corresponding to $\sss_{k-1}$ is $1$ greater than that of $\yy_{k-1}$. Based on this observation, we propose the following interpolated least squares model
\begin{equation}\label{VLS}
	\min_{\alpha>0} L(\alpha):=\|\alpha^{m_k} \sss_{k-1} -\alpha^{m_k-1} \yy_{k-1}\|_{2}^{2},
\end{equation}
where $m_k\in[0, 1]$. Obviously, $\alpha_{k}^{BB1}$ and $\alpha_{k}^{BB2}$ correspond to $m_k=1$ and $m_k=0$ in \eqref{VLS}, respectively. We now consider $m_k\in(0, 1]$. Taking the derivative of the objective function $L(\alpha)$ in \eqref{VLS} with respect to $\alpha$, we have 
\begin{equation}\label{equ:LA}
	L'(\alpha)=2\alpha^{2m_{k}-3}\big[m_{k}\sss_{k-1}^{\T}\sss_{k-1}\alpha^{2}-(2m_{k}-1)\sss_{k-1}^{\T}\yy_{k-1}\alpha+(m_{k}-1)\yy_{k-1}^{\T}\yy_{k-1}\big].
\end{equation}
The first-order optimality condition of problem \eqref{VLS} is $L'(\alpha)=0$. Based on this, along with condition \eqref{equ:LA} and $\alpha\ne0$, we know that solving \eqref{VLS} is equivalent to solving the following quadratic equation
\begin{equation}\label{quadratic equ}
	m_k\sss_{k-1}^{\T}\sss_{k-1}\alpha^{2}-(2m_k-1)\sss_{k-1}^{\T}\yy_{k-1}\alpha+(m_k-1)\yy_{k-1}^{\T}\yy_{k-1}=0.
\end{equation} 
Let 
$$\theta_{k}=\angle(\sss_{k-1}, \yy_{k-1}).$$
Then we have 
\begin{equation}\label{equ:cosk}
	\cos^{2}\theta_{k}=\frac{(\sss_{k-1}^{\T}\yy_{k-1})^2}{\|\sss_{k-1}\|_{2}^{2}\|\yy_{k-1}\|_{2}^{2}}.
\end{equation}
According to the discriminant of the quadratic equation \eqref{quadratic equ} as follows  $$\Delta=(2m_{k}-1)^2(\sss_{k-1}^{\T}\yy_{k-1})^2-4m_{k}(m_{k}-1)\sss_{k-1}^{\T}\sss_{k-1}\yy_{k-1}^{\T}\yy_{k-1},$$ 
if
\begin{equation}\label{Discriminant}
	\cos^{2}\theta_{k}\ge\frac{4m_k(m_k-1)}{(2m_k-1)^{2}},
\end{equation}
then the equation \eqref{quadratic equ} has real roots. Under the condition of $m_k\in(0, 1]$, we know that \eqref{Discriminant} always holds. Let the two real roots of \eqref{quadratic equ} be $\alpha_{k,1}$ and $\alpha_{k,2}$, and $\alpha_{k,1}\le\alpha_{k,2}$. Then we have 
\begin{equation*}
	\alpha_{k,1}\alpha_{k,2}=\frac{(m_k-1)\yy_{k-1}^{\T}\yy_{k-1}}{m_k\sss_{k-1}^{\T}\sss_{k-1}}\le 0.
\end{equation*}
Therefore, $\alpha_{k,1}$ and $\alpha_{k,2}$ have different signs. We take positive $\alpha_{k,2}$ as $\alpha_{k}^{PBB}$, and have
\begin{equation}\label{VBB}
	\alpha_{k}^{PBB}=\frac{(2m_k-1)\sss_{k-1}^{\T}\yy_{k-1}+\sqrt{\big((2m_k-1)\sss_{k-1}^{\T}\yy_{k-1}\big)^{2}-4m_k(m_k-1)\sss_{k-1}^{\T}\sss_{k-1}\yy_{k-1}^{\T}\yy_{k-1}}}{2m_k\sss_{k-1}^{\T}\sss_{k-1}},
\end{equation} 
where $m_k\in(0, 1]$. By choosing different parameter $m_k$, we obtain a family of BB-like step sizes $(\alpha_{k}^{PBB})^{-1}$. We formally call the gradient descent method with step size $(\alpha_{k}^{PBB})^{-1}$ the PBB method. From the perspective of interpolated least squares, the PBB method achieves interpolation between the classical BB1 and BB2 step sizes by selecting interpolation parameter $m_{k}$.

Some properties of $\alpha_{k}^{PBB}$ are summarized in Theorem \ref{theorem1} as follows.
\begin{theorem}\label{theorem1}
	Let $m_k\in(0, 1]$.  Then
	$\alpha_{k}^{PBB}$ is monotonically decreasing with respect to $m_k$, and 
	\begin{equation*}
	\begin{split}
	&\alpha_{k}^{PBB}\in[\alpha_{k}^{BB1},\alpha_{k}^{BB2}).\\
	\end{split}
	\end{equation*}
\end{theorem}
\begin{proof}
	Based on  $\sss_{k-1}^{\T}\yy_{k-1}\le\|\sss_{k-1}\|_{2}\|\yy_{k-1}\|_{2}$, we calculate the derivative of $\alpha_{k}^{PBB}$ with respect to $m_k$ as follows
	\begin{equation}\label{VBB'}
		\begin{split}
		\frac{d}{dm_k}\big(\alpha_{k}^{PBB}(m_k)\big)&=\frac{1}{2m_k^2}\sss_{k-1}^{\T}\yy_{k-1}+\frac{\frac{2m_k-1}{2m_k}\frac{1}{2m_k^2}(\sss_{k-1}^{\T}\yy_{k-1})^2-\frac{1}{2m_k^2}\|\sss_{k-1}\|_{2}^{2}\|\yy_{k-1}\|_{2}^{2}}{\sqrt{(\frac{2m_k-1}{2m_k})^2(\sss_{k-1}^{\T}\yy_{k-1})^2-\frac{m_k-1}{m_k}\|\sss_{k-1}\|_{2}^{2}\|\yy_{k-1}\|_{2}^{2}}}\\
		&\le\frac{1}{2m_k^2}\sss_{k-1}^{\T}\yy_{k-1}+\frac{\frac{2m_k-1}{2m_k}\frac{1}{2m_k^2}(\sss_{k-1}^{\T}\yy_{k-1})^2-\frac{1}{2m_k^2}(\sss_{k-1}^{\T}\yy_{k-1})^{2}}{\sqrt{(\frac{2m_k-1}{2m_k})^2(\sss_{k-1}^{\T}\yy_{k-1})^2}}.
 		\end{split}
	\end{equation} 	
In the case $m_k\in(0,\frac{1}{2})$, 
\begin{equation}\label{le1}
	\begin{split}
	&\frac{1}{2m_k^2}\sss_{k-1}^{\T}\yy_{k-1}+\frac{\frac{2m_k-1}{2m_k}\frac{1}{2m_k^2}(\sss_{k-1}^{\T}\yy_{k-1})^2-\frac{1}{2m_k^2}(\sss_{k-1}^{\T}\yy_{k-1})^{2}}{\sqrt{(\frac{2m_k-1}{2m_k})^2(\sss_{k-1}^{\T}\yy_{k-1})^2}}\\
	&=\frac{1}{2m_k^2}\sss_{k-1}^{\T}\yy_{k-1}-\frac{1}{2m_k^2}\sss_{k-1}^{\T}\yy_{k-1}+\frac{1}{2m_k^2}\frac{2m_k}{2m_k-1}\sss_{k-1}^{\T}\yy_{k-1}\\
	&=\frac{1}{m_k(2m_k-1)}\sss_{k-1}^{\T}\yy_{k-1}\\
	&<0.
	\end{split}
\end{equation}	
In the case $m_k\in(\frac{1}{2}, 1]$, 
\begin{equation}\label{le2}
\begin{split}
&\frac{1}{2m_k^2}\sss_{k-1}^{\T}\yy_{k-1}+\frac{\frac{2m_k-1}{2m_k}\frac{1}{2m_k^2}(\sss_{k-1}^{\T}\yy_{k-1})^2-\frac{1}{2m_k^2}(\sss_{k-1}^{\T}\yy_{k-1})^{2}}{\sqrt{(\frac{2m_k-1}{2m_k})^2(\sss_{k-1}^{\T}\yy_{k-1})^2}}\\
&=\frac{1}{2m_k^2}\sss_{k-1}^{\T}\yy_{k-1}+\frac{1}{2m_k^2}\sss_{k-1}^{\T}\yy_{k-1}-\frac{1}{2m_k^2}\frac{2m_k}{2m_k-1}\sss_{k-1}^{\T}\yy_{k-1}\\
&=\frac{m_{k}-1}{m_k^{2}(2m_k-1)}\sss_{k-1}^{\T}\yy_{k-1}\\
&\le 0.
\end{split}
\end{equation}		
Combining \eqref{VBB'}, \eqref{le1}, and \eqref{le2}, we know that $\alpha_{k}^{PBB}$ is monotonically decreasing with respect to $m_k$. Direct calculation shows that $\alpha_{k}^{BB1}$, $\sqrt{\alpha_{k}^{BB1}\alpha_{k}^{BB2}}$, and $\alpha_{k}^{BB2}$ correspond to $m_k=1$, $\frac{1}{2}$, and $0$. We complete the proof.
\end{proof}
From \eqref{quadratic equ}, we consider a quadratic function 
\begin{equation*}
	\phi(\alpha)=m_k\sss_{k-1}^{\T}\sss_{k-1}\alpha^{2}-(2m_k-1)\sss_{k-1}^{\T}\yy_{k-1}\alpha+(m_k-1)\yy_{k-1}^{\T}\yy_{k-1}
\end{equation*}
with $\sss_{k-1}^{\T}\sss_{k-1}=2$, $\sss_{k-1}^{\T}\yy_{k-1}=3$, and $\yy_{k-1}^{\T}\yy_{k-1}=9$. Figure \ref{fig:roots} plots the graph of $\phi(\alpha)$ for various $m_k\in[0, 1]$. 
\begin{figure}[h!]
	\centering
	\subfigure{\includegraphics[width=0.8\linewidth]{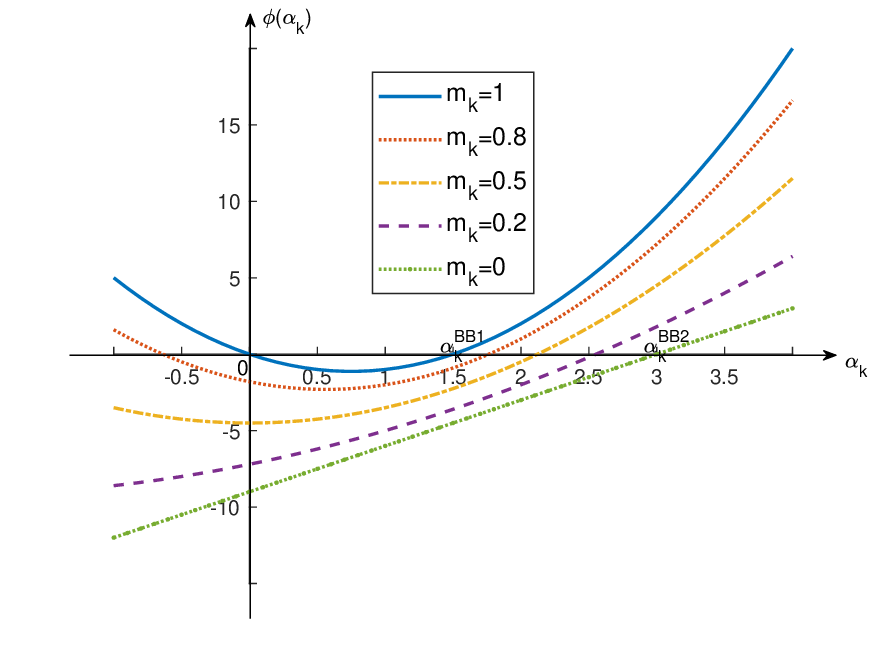}}
	\caption{\textit{The horizontal axis is $\alpha_{k}$, and the vertical axis is the value of $\phi(\alpha_{k})$.}}
	\label{fig:roots}
\end{figure}
We can clearly see that as $m_k$ decreases from $1$ to $0$, the positive zero point of $\phi(\alpha)$ gradually increases from $\alpha_{k}^{BB1}$ to $\alpha_{k}^{BB2}$. When $m_k=0$, $\phi(\alpha)$ is a linear function with zero at $\alpha_{k}^{BB2}$, as shown in the figure. This example graphically describes the effect of $\alpha_{k}^{PBB}$ interpolating between two classical $\alpha_{k}^{BB1}$ and $\alpha_{k}^{BB2}$ via the interpolation parameter $m_{k}$.
\section{Convergence analysis}\label{sec:convergence}
In this section, we analyze the convergence of the PBB method for strictly convex quadratic minimization problems. 

Due to the results in \cite{Forsythe1968asymptoticdirectionsthes,Yuan2006newstepsizesteepest}, the behavior of gradient descent method for higher dimensional problems are essential the same as it for two dimensional cases. Therefore, we consider minimizing a two dimensional quadratic convex function as follows
\begin{equation}\label{quadratic}
\min_{\x\in\mathbb{R}^{2}} f(\x)=\frac{1}{2}\x^{\T}A\x-\mathbf{b}^{\T}\x,
\end{equation} 
where $\mathbf{b}\in\mathbb{R}^{2}$ and   
\begin{equation*}
A=\begin{bmatrix}
\lambda & 0\\
0 & 1
\end{bmatrix}
\end{equation*}
with $\lambda>1$. We denote $\g_{k-1}=(\g_{k-1}^{(1)}, \ \g_{k-1}^{(2)})^{\T}$ with $\g_{k-1}^{(i)}\neq 0$ for $i=1,2$. In this case, from the result in Theorem \ref{theorem1}, we have
\begin{equation}\label{equ:inequality}
1\le\alpha_{k}^{BB1}\le\alpha_{k}^{PBB}\le\alpha_{k}^{BB2}\le\lambda.
\end{equation} 
From \eqref{quadratic}, we have
\begin{equation}\label{gkk}
\g_{k+1}=A\x_{k+1}-\mathbf{b}.
\end{equation}
It follows from \eqref{gradient method} and \eqref{gkk} that
\begin{equation}\label{equ:gkkk}
\g_{k+1}=\big[I-(\alpha_{k}^{PBB})^{-1}A\big]\g_{k}.
\end{equation}
From \eqref{equ:inequality} and \eqref{equ:gkkk}, we know that 
\begin{equation}\label{equ:egk}
(\g_{k+1}^{(2)})^2=\Big(\frac{\alpha_{k}^{PBB}-1}{\alpha_{k}^{PBB}}\Big)^2(\g_{k}^{(2)})^2,
\end{equation}
and
\begin{equation}\label{equ:egk2}
(\g_{k+1}^{(1)})^2=\Big(\frac{\alpha_{k}^{PBB}-\lambda}{\alpha_{k}^{PBB}}\Big)^2(\g_{k}^{(1)})^2.
\end{equation}
It follows from \eqref{equ:inequality} and \eqref{equ:egk} that $|\g_{k}^{(2)}|$ is monotonically decreasing. We now focus on analyzing the change of $|\g_{k}^{(1)}|$ because it may show a non-monotonic phenomenon. Let
\begin{equation}\label{equ:ek}
\epsilon_{k}=\frac{(\g_{k-1}^{(1)})^{2}}{(\g_{k-1}^{(2)})^{2}}>0.
\end{equation}   
Since $\sss_{k-1}=\x_{k}-\x_{k-1}=-(\alpha_{k-1}^{PBB})^{-1}\g_{k-1}$ and $\yy_{k-1}=\g_{k}-\g_{k-1}=-(\alpha_{k-1}^{PBB})^{-1}A\g_{k-1}$, according the definition of $\alpha_{k}^{PBB}$ in  \eqref{VBB}, we have 
\begin{align}\label{equ:nRBB}
\begin{split}
\alpha_{k}^{PBB}=\frac{(2m_{k}-1)(\lambda\epsilon_{k}+1)+\sqrt{(\lambda\epsilon_{k}+1)^2+4m_{k}(1-m_{k})(\lambda-1)^2\epsilon_{k}}}{2m_{k}(\epsilon_{k}+1)}
\end{split}
\end{align}
with $m_{k}\in(0,1]$.

From \eqref{equ:egk}, \eqref{equ:ek}, and given that
\begin{equation}\label{equ:gknorm}
\|\g_{k+1}\|_{2}^{2}=(\g_{k+1}^{(1)})^2(1+\epsilon_{k+2}),	
\end{equation}  
we know that if $\epsilon_{k+1}\le1$, then $\|\g_{k+1}\|_{2}<\|\g_{k}\|_{2}$; otherwise, $\|\g_{k+1}\|_{2}$ may be greater than $\|\g_{k}\|_{2}$. With this clue, we focus our analysis on the variation pattern of $\epsilon_{k}$.
It follows from \eqref{equ:egk}, \eqref{equ:egk2}, \eqref{equ:ek}, and \eqref{equ:nRBB} that 
\begin{equation}\label{equ:dynamic}
\epsilon_{k+2}=e(\epsilon_{k})^2\epsilon_{k+1},
\end{equation}
where
\begin{equation}\label{equ:eek}
e(\epsilon_{k})=\frac{(2m_{k}-1)(\lambda\epsilon_{k}+1)+\sqrt{(\lambda\epsilon_{k}+1)^2+4m_{k}(1-m_{k})(\lambda-1)^2\epsilon_{k}}-2m_{k}\lambda(\epsilon_{k}+1)}{(2m_{k}-1)(\lambda\epsilon_{k}+1)+\sqrt{(\lambda\epsilon_{k}+1)^2+4m_{k}(1-m_{k})(\lambda-1)^2\epsilon_{k}}-2m_{k}(\epsilon_{k}+1)}.
\end{equation}
\begin{proposition}\label{prop1}
	In \eqref{equ:eek}, the denominator of $e(\epsilon_{k})$ is greater than $0$.
\end{proposition}
\begin{proof}
	Because $m_{k}\in(0,1]$, then we have
	\begin{equation}
	\begin{split}
	&(2m_{k}-1)(\lambda\epsilon_{k}+1)+\sqrt{(\lambda\epsilon_{k}+1)^2+4m_{k}(1-m_{k})(\lambda-1)^2\epsilon_{k}}-2m_{k}(\epsilon_{k}+1)\\
	&\ge(2m_{k}-1)(\lambda\epsilon_{k}+1)+(\lambda\epsilon_{k}+1)-2m_{k}(\epsilon_{k}+1)\\
	&=2m_{k}(\lambda-1)\epsilon_{k}\\
	&>0.
	\end{split}	
	\end{equation}
	We complete the proof. 
\end{proof}
We always assume that $\varepsilon_{k+1}>1$, because if this assumption is not true all the time, then $\|\g_{k+1}\|_{2}\le\|\g_{k}\|_{2}$ for all $k\ge1$, indicating that $\|\g_{k}\|_{2}$ is monotonically decreasing.
\begin{lemma}\label{lemma:lem1}
	Assume that $\epsilon_{k+1}>1$. Then there exists an integer $N\ge1$ such that
	\begin{equation}
	\epsilon_{k+N+1}\le 1.
	\end{equation} 
\end{lemma}
\begin{proof}
	Assume, for the purpose of deriving a contradiction, that for all $N\ge1$,
	\begin{equation}\label{equ:assume}
	\epsilon_{k+N+1}>1.
	\end{equation}
	We denote
	\begin{equation}\label{equ:delta}
		\delta=\frac{1}{\epsilon_{k+N+2}}\in(0,1).
	\end{equation} 
	According to \eqref{equ:dynamic} and \eqref{equ:assume}, we have
	\begin{align}\label{equ:lizi}
	\begin{split}
	0<\epsilon_{k+N+3}-1&=e(\epsilon_{k+N+1})^2\epsilon_{k+N+2}-1\\
	&=\big[e(\epsilon_{k+N+1})^2-\delta\big]\epsilon_{k+N+2}.
	\end{split}
	\end{align}  
	It follows from \eqref{equ:assume} and \eqref{equ:lizi} that 
	\begin{equation}\label{equ:lizi2}
	e(\epsilon_{k+N+1})^2>\delta.
	\end{equation}
	From \eqref{equ:eek} and Proposition \ref{prop1}, we know that the solution set of inequality \eqref{equ:lizi2} about $m_{k}$ is
	\begin{equation}\label{equ:beta}
		m_{k}<-\frac{(1-\sqrt{\delta})(1+\sqrt{\delta}\lambda\epsilon_{k+N+1})}{(\lambda-1)(1+\delta\epsilon_{k+N+1})},
	\end{equation}
or
	\begin{equation}\label{equ:eta}
		m_{k}>\frac{(1+\sqrt{\delta})(\lambda\sqrt{\delta}\epsilon_{k+N+1}-1)}{(1+\delta\epsilon_{k+N+1})(\lambda-1)}.
	\end{equation}
The right-hand side of \eqref{equ:beta} is negative, which contradicts $m_{k}\in(0,1]$. Thus, we only consider the case \eqref{equ:eta}. Setting 
\begin{equation}\label{equ:c}
	c=\sqrt{\delta}\epsilon_{k+N+1},
\end{equation}
if inequality \eqref{equ:eta} holds universally, then it implies that $c\in(0,\frac{1}{\lambda})$. Nevertheless, it follows from \eqref{equ:assume} and \eqref{equ:delta} that $c\in(0,\frac{1}{\lambda})$ cannot be guaranteed, indicating that assumption \eqref{equ:assume} does not hold. We complete the proof.	
	
\end{proof}
\begin{remark}
In the proof of Lemma \ref{lemma:lem1}, if $\lambda$ is a large real number, then the right-hand side of \eqref{equ:eta} is dominated by $c$. Therefore, if $c>1$, then the inequality in \eqref{equ:eta} is likely to be invalid, that is, \eqref{equ:lizi2} has no solution. This means that in three consecutive iterations, at least one iteration step is decreasing. It is worth noting that \eqref{equ:dynamic} is a second-order nonlinear difference dynamical system about $\epsilon_{k}$, and its analytical expression is extremely complicated. A special case of \eqref{equ:dynamic} is: $\epsilon_{k+2}=\frac{\epsilon_{k+1}}{(\epsilon_{k})^2}$, which corresponds to the case of $m_{k}=1$, that is, the second-order differential dynamical system in the BB1 method. The analytical solution of this special dynamical system can be expressed explicitly, and the result can be seen in \cite{Barzilai1988TwoPointStep}. Based on this perspective, for analyzing the convergence of gradient descent methods and even designing effective spectral gradient step sizes, we should consider certain properties of the corresponding second-order difference dynamical system. This approach may open up new avenues for theoretical analysis of optimization algorithms.
\end{remark}
\begin{theorem}
	Let $f(\x)$ be the strictly convex quadratic function \eqref{quadratic}. Let $\{\x_{k}\}$ be the sequence generated by the PBB method. Then, either $\g_{k}=\mathbf{0}$ for some finite $k$, or the sequence $\{\|\g_{k}\|_{2}\}$ converges to zero at least $R$-linearly.
\end{theorem}
\begin{proof}
	We only need to consider the case that $\g_{k}\neq\mathbf{0}$ for all $k$. We assume that $\epsilon_{k+1}>1$ holds for some $k>1$. From \eqref{equ:egk}, \eqref{equ:gknorm}, and Lemma \ref{lemma:lem1}, we know that there exists an integer $N\ge1$ such that  
	\begin{align*}
	\begin{split}
	\|\g_{k+N}\|_{2}^{2}&=(\g_{k+N}^{(2)})^2(1+\epsilon_{k+N+1})\\
	&=\prod_{j=0}^{N-1}(\xi_{k+j})^2(\g_{k}^{(2)})^2(1+\epsilon_{k+N+1})\\
	&=\prod_{j=0}^{N-1}(\xi_{k+j})^2\frac{1+\epsilon_{k+N+1}}{1+\epsilon_{k+1}}\|\g_{k}\|_{2}^{2}\\
	&<\prod_{j=0}^{N-1}(\xi_{k+j})^2\|\g_{k}\|_{2}^{2},
	\end{split}
	\end{align*}
	where each $\xi_{k+j}=\Big(\frac{\alpha_{k}^{PBB}-1}{\alpha_{k}^{PBB}}\Big)^2\in(0,1)$. Therefore, the sequence $\{\g_{k}\}$ converges to zero $R$-linearly. This completes our proof.	
\end{proof}
\begin{remark}
	For the $n$ dimensional case, we can consider
	\begin{equation*}
	A=\begin{bmatrix}
	A_{1} & \mathbf{0}\\
	\mathbf{0} & 1
	\end{bmatrix},
	\end{equation*}
	where the sub-matrix 
	$$A_{1}=\text{diag}(\lambda_{1},\lambda_{2},\ldots,\lambda_{n-1}),$$
	where $\lambda=\lambda_{1}\ge \lambda_{2}\ge\ldots\ge \lambda_{n-1}>1$, and denote $\g_{k-1}=\big((\g_{k-1}^{(1)})^{\T}, \ \g_{k-1}^{(2)}\big)^{\T}$ with $\g_{k-1}^{(1)}\in\mathbb{R}^{n-1}(\neq\mathbf{0})$,  $\g_{k-1}^{(2)}\neq 0$, and let
	\begin{equation}\label{equ:nepsilon}
	\epsilon_{k}=\frac{\|\g_{k-1}^{(1)}\|_{2}^{2}}{(\g_{k-1}^{(2)})^{2}}.
	\end{equation}
	We can derive the conclusion of Lemma \ref{lemma:lem1}, thereby proving the $R$-linear convergence of the PBB method. For the problem \eqref{quadratic}, we know that  $\yy_{k-1}=A\sss_{k-1}$ and  $\g_{k-1}^{\T}\g_{k-1}\g_{k-1}^{\T}A^2\g_{k-1}\ge(\g_{k-1}^{\T}A\g_{k-1})^2$ hold. Then we have	
	\begin{equation}
	\begin{split}
	\alpha_{k}^{PBB}&=\frac{(2m_k-1)\sss_{k-1}^{\T}\yy_{k-1}+\sqrt{\big((2m_k-1)\sss_{k-1}^{\T}\yy_{k-1}\big)^{2}-4m_k(m_k-1)\sss_{k-1}^{\T}\sss_{k-1}\yy_{k-1}^{\T}\yy_{k-1}}}{2m_k\sss_{k-1}^{\T}\sss_{k-1}}\\
	&\ge \frac{\frac{2m_k-1}{2m_{k}}\g_{k-1}^{\T}A\g_{k-1}+\sqrt{\big(\frac{2m_k-1}{2m_{k}}\g_{k-1}^{\T}A\g_{k-1}\big)^{2}+\frac{1-m_{k}}{m_{k}}(\g_{k-1}^{\T}A\g_{k-1})^2}}{\g_{k-1}^{\T}\g_{k-1}}\\
	&=\frac{\g_{k-1}^{\T}A\g_{k-1}}{\g_{k-1}^{\T}\g_{k-1}}.
	\end{split}
	\end{equation} According to Property B in \cite{Li2023NoteRLinear}, readers may also note that the PBB method exhibits $R$-linear convergence in the $n$ dimensional case.
\end{remark}

\section{An analysis of the ABB method}\label{sec:ABB}
In this section, the principle of adaptive alternating BB step sizes is analyzed to provide insights for the subsequent selection of the interpolation parameter $m_{k}$. From the result of Theorem \ref{theorem1}, we know that $\alpha_{k}^{PBB}\rightarrow\alpha_{k}^{BB1}$ as $m_{k}\rightarrow1$ and $\alpha_{k}^{PBB}\rightarrow\alpha_{k}^{BB2}$ as $m_{k}\rightarrow0$. By selecting the parameter $m_{k}$ from $[0,1]$, $\alpha_{k}^{PBB}$ interpolates between  $\alpha_{k}^{BB1}$ and $\alpha_{k}^{BB2}$. If $m_{k}\in\{0,1\}$, then $\alpha_{k}^{PBB}$ corresponds to $\alpha_{k}^{ABB}$ \cite{Zhou2006GradientMethodsAdaptive}. In this sense, the ABB method can be regarded as a special case of the PBB method. 

The ABB method adaptively selects either the BB1 or BB2 step size via the following threshold-based strategy 
\begin{equation}\label{ABB}
\alpha_{k}^{ABB}=\begin{cases}
\alpha_{k}^{BB2},\quad \text{if}\quad  \alpha_{k}^{BB1}/\alpha_{k}^{BB2}<\eta,\\
\alpha_{k}^{BB1},\quad \text{otherwise},
\end{cases}
\end{equation}
where $\eta\in(0,1)$. From \eqref{equ:cosk}, we know that  $$\cos^{2}\theta_{k}=\alpha_{k}^{BB1}/\alpha_{k}^{BB2}.$$ The idea of the alternating strategy in  \eqref{ABB} is to take a long BB1 step size when $\cos^{2}\theta_{k}\rightarrow1$ and a short BB2 step size when $\cos^{2}\theta_{k}\rightarrow0$. In the gradient descent method, the most essential issue is the process of eliminating the gradient component. Based on this point of view, we present an analysis of the \text{ABB} method. 

We continue to consider the two dimensional problem \eqref{quadratic} while adhering to the settings established in Section \ref{sec:convergence}. Then we have 
\begin{equation}\label{ratio}
\alpha_{k}^{BB1}/\alpha_{k}^{BB2}=\frac{(\lambda\epsilon_{k}+1)^2}{(\epsilon_{k}+1)(\lambda^{2}\epsilon_{k}+1)}.
\end{equation} 
If $\alpha_{k}^{BB1}/\alpha_{k}^{BB2}<\eta$ holds, then we have
\begin{equation}\label{epsilon}
(\lambda\epsilon_{k}+1)^2<\eta(\epsilon_{k}+1)(\lambda^{2}\epsilon_{k}+1).
\end{equation}
After simplifying \eqref{epsilon}, we obtain the following quadratic inequality in $\epsilon_{k}$ 
\begin{equation}\label{epsilonfunc}
\Phi(\epsilon_{k}):=\lambda^{2}(1-\eta)\epsilon_{k}^{2}+\big[2\lambda-\eta(1+\lambda_{k}^{2})\big]\epsilon_{k}+1-\eta<0.
\end{equation} 
If 
\begin{equation}\label{condition}
\eta>\frac{4\lambda}{(1+\lambda)^{2}},
\end{equation}
then the equation $\Phi(\epsilon_{k})=0$  has two positive real roots $\epsilon_{k}^{1}$ and $\epsilon_{k}^{2}$. Let $\epsilon_{k}^{1}<\epsilon_{k}^{2}$. Then  
\begin{equation*}\label{solutions}
\epsilon_{k}^{1}=\frac{\eta(1+\lambda^{2})-2\lambda-(\lambda^2-1)\sqrt{\eta\big(\eta-\frac{4\lambda}{(1+\lambda)^2}\big)}}{2\lambda^{2}(1-\eta)},\quad \epsilon_{k}^{2}=\frac{\eta(1+\lambda^{2})-2\lambda+(\lambda^2-1)\sqrt{\eta\big(\eta-\frac{4\lambda}{(1+\lambda)^2}\big)}}{2\lambda^{2}(1-\eta)}.
\end{equation*}
Therefore, if $\epsilon_{k}\in(\epsilon_{k}^{1}, \epsilon_{k}^{2})$, then we know that  \eqref{epsilonfunc} holds. Let $\lambda\gg1$. If $\frac{4\lambda}{(1+\lambda)^2}\to 0$, then we have the following properties (due to $\epsilon_{k}^{1}\epsilon_{k}^{2}=\frac{1}{\lambda^2}$) 
\begin{equation}\label{limita1a2}
\epsilon_{k}^{1}\to \frac{1-\eta}{\eta\lambda^2},\quad \epsilon_{k}^{2}\to\frac{\eta}{1-\eta}.
\end{equation}
If we further restrict $\eta\in(0,0.5]$ as suggested in \cite{Zhou2006GradientMethodsAdaptive}, then $\frac{\eta}{1-\eta}\in(0,1]$. It follows from the results in \eqref{limita1a2} that 
\begin{equation}\label{results}
\begin{split} &{\lim_{\epsilon_{k}\to\epsilon_{k}^{2}}\alpha_{k}^{BB1}=\frac{\lambda\frac{\eta}{1-\eta}+1}{\frac{\eta}{1-\eta}+1}=\eta(\lambda-1)+1},\quad{\lim_{\epsilon_{k}\to\epsilon_{k}^{2}}\alpha_{k}^{BB2}=\frac{\lambda^2\frac{\eta}{1-\eta}+1}{\lambda\frac{\eta}{1-\eta}+1}\rightarrow\lambda},\\ &{\lim_{\epsilon_{k}\to\epsilon_{k}^{1}}\alpha_{k}^{BB1}=\frac{\lambda\frac{1-\eta}{\eta\lambda^{2}}+1}{\frac{1-\eta}{\eta\lambda^{2}}+1}\rightarrow 1},\quad{\lim_{\epsilon_{k}\to\epsilon_{k}^{1}}\alpha_{k}^{BB2}=\frac{\lambda^2\frac{1-\eta}{\eta\lambda^{2}}+1}{\lambda\frac{1-\eta}{\eta\lambda^{2}}+1}\rightarrow\frac{1}{\eta}}<\lambda (\text{due to \eqref{condition}}).
\end{split}
\end{equation}
When $\eta\in(0, 0.5]$,  according to the definition $\epsilon_{k}$ in \eqref{equ:ek} and the results in \eqref{limita1a2} and \eqref{results}, we discuss three cases as follows.
\begin{itemize}
	\item[($\text{\uppercase\expandafter{\romannumeral1}}$)] $\epsilon_{k}\in(\epsilon_{k}^{1},\ \epsilon_{k}^{2})$, i.e., $\alpha_{k}^{BB1}/\alpha_{k}^{BB2}<\eta$. This scenario indicates that, compared to the gradient component $|\g_{k-1}^{(2)}|$, the gradient component $|\g_{k-1}^{(1)}|$ corresponding to the eigenvalue $\lambda$ has not been completely eliminated. Selecting $\alpha_{k}^{BB2}$ is conducive to approximating $\lambda$, thereby achieving the purpose of quickly eliminating $|\g_{k}^{(1)}|$.
	\item[($\text{\uppercase\expandafter{\romannumeral2}}$)] $\epsilon_{k}\in[0,\ \epsilon_{k}^{1}]$. This scenario indicates that $|\g_{k-1}^{(1)}|$ is already very small or close to being completely eliminated. Choosing $\alpha_{k}^{BB1}$ is conducive to approaching eigenvalue $1$ and quickly eliminating $|\g_{k}^{(2)}|$.
	\item[($\text{\uppercase\expandafter{\romannumeral3}}$)] $\epsilon_{k}\in[\epsilon_{k}^{2},\infty)$. This case indicates that the size of  $|\g_{k-1}^{(1)}|$ is equivalent to or exceeds that of $|\g_{k-1}^{(2)}|$. In this case, it is still beneficial to select $\alpha_{k}^{BB2}$.   
\end{itemize}
The motivation behind implementing alternating step sizes, as described in the above three scenarios, lies in enhancing the algorithm's  capability to probe the spectrum of the Hessian matrix. In the ABB method, the alternating threshold $\eta$ is a user-specified constant, typically required to satisfy $\eta\in(0,0.5]$. This approach exhibits two limitations: firstly, we possess no prior knowledge regarding which specific value of $\eta$ would be most appropriate; secondly, as illustrated in the aforementioned three cases, the condition $\eta\in(0,0.5]$ fails to cover the third case. Furthermore, by letting $\lambda\gg 1$, we obtain favorable results in \eqref{results}. Nevertheless, in practical problems, $\lambda$ is finite. Consequently, $\alpha_{k}^{BB2}$ cannot approximate $\lambda$ well, nor can $\alpha_{k}^{BB1}$ effectively approximate $1$. To address these limitations, the ABBmin method  \cite{Frassoldati2008Newadaptivestepsize} employs the maximum value of $\alpha_{k}^{BB2}$ from the recent several iterations to approximate $\lambda$. Additionally, an alternating threshold $\eta\in(0.5,1)$ is set to maximize coverage of Case $\text{\uppercase\expandafter{\romannumeral3}}$. From the analysis of these three cases, it can be seen that the three-term alternating strategy seems to be a more efficient scheme in gradient descent methods. We now aim to partially mitigate these deficiencies of ABB method by adaptively selecting the interpolation parameter $m_{k}$ based on problem characteristics.

\begin{remark}
	In general, condition \eqref{condition} is easily satisfied for $\lambda>1$. Locally, the BB method is dominated by a two dimensional subspace \cite{Huang2022accelerationBarzilaiBorweinmethod}, where the eigenvalues $\lambda$ and $1$ represent the two eigenvalues of the subspace. When the gradient component corresponding to the subspace is eliminated, the BB method is gradually dominated by the next two dimensional subspace.  
\end{remark}

\section{Selecting $m_{k}$ based on \text{ABB} scheme}\label{sec:mk}
Following the philosophy of \text{ABB} method and using the variation information between two consecutive iterations, we first propose an adaptive parameter $\zeta_{k}$ capable of capturing local information of the problem as follows
\begin{equation}\label{zeta}
	\zeta_{k}=\cos^{2}\theta_{k}\frac{\cos^{2}\theta_{k}}{\cos^{2}\theta_{k-1}}\in(0,\infty).
\end{equation}
If $\cos^{2}\theta_{k}\rightarrow1$ and $\cos^{2}\theta_{k-1}\rightarrow0$, then $\frac{\cos^{2}\theta_{k}}{\cos^{2}\theta_{k-1}}\rightarrow\infty$. If $\cos^{2}\theta_{k}\rightarrow0$ and $\cos^{2}\theta_{k-1}\rightarrow1$, then $\frac{\cos^{2}\theta_{k}}{\cos^{2}\theta_{k-1}}\rightarrow0$. When $\cos^2\theta_{k}\rightarrow1$, it implies that the gradient $\g_{k-1}$ approximates an eigenvector of the Hessian matrix, indicating that the current local search direction is favorable and selecting a long step size is recommended; whereas choosing a short step size is reasonable otherwise. Note that  $\frac{\cos^{2}\theta_{k}}{\cos^{2}\theta_{k-1}}$ describes only the relative variation information of the search direction between two adjacent iterations, without using the current size of $\cos^{2}\theta_{k}$, so we multiply $\cos^{2}\theta_{k}$ in front of $\frac{\cos^{2}\theta_{k}}{\cos^{2}\theta_{k-1}}$. Specifically, we discuss the principle of $\zeta_{k}$ from the following three cases:
\begin{enumerate}
	\item[(1)] $\frac{\cos^2\theta_{k}}{\cos^2\theta_{k-1}}<1\Longrightarrow0<\zeta_{k}<1$, 
	\item[(2)] $1<\frac{\cos^2\theta_{k}}{\cos^2\theta_{k-1}}<\frac{1}{\cos^2\theta_{k}}\Longrightarrow0<\zeta_{k}<1$,
	\item[(3)] $\frac{\cos^2\theta_{k}}{\cos^2\theta_{k-1}}>1$ and  $\frac{\cos^2\theta_{k}}{\cos^2\theta_{k-1}}>\frac{1}{\cos^2\theta_{k}}\Longrightarrow\zeta_{k}>1$.
\end{enumerate}
In the first case, $\cos^2\theta_{k}<\cos^2\theta_{k-1}$, so it is reasonable to choose a short step; in the second case, even   $\cos^2\theta_{k}>\cos^2\theta_{k-1}$, but if $\cos^2\theta_{k}$ is very small, resulting in $\frac{\cos^2\theta_{k}}{\cos^2\theta_{k-1}}<\frac{1}{\cos^2\theta_{k}}$, this means that $\g_{k-1}$ deviates from one eigenvector of Hessian, then we have to choose a short step; the last case shows that $\cos^2\theta_{k}$ is large, and a long step should be chosen.  
\begin{remark}
	The parameter scheme \eqref{zeta} originates from the idea of \text{ABB}, but is different from the operation in \text{ABB}. One advantage of scheme \eqref{zeta} is that it is adaptive and does not require user to set the threshold value $\eta$. We can regard $\frac{\cos^2\theta_{k}}{\cos^2\theta_{k-1}}$ and $\frac{1}{\cos^2\theta_{k}}$ in \eqref{zeta} as $\cos^2\theta_{k}$ and $\eta$ in \text{ABB} \eqref{ABB}, respectively.
\end{remark}

Because the parameter $m_{k}\in(0,1]$, we consider the following transformation
\begin{equation}\label{mk}
	m_{k}=\frac{(\zeta_{k})^q}{\alpha_{k}^{BB1}+(\zeta_{k})^q},
\end{equation} 
where $q\ge1$ is an integer and acts as a scaling factor. That is, when $\zeta_{k}\in(0, 1)$, $m_{k}$ will quickly approach $0$ through the action of $q$; conversely, when $\zeta_{k}>1$, $m_{k}$ will quickly approach $1$ through the action of $q$. In the later experimental section, we will test the performance of the PBB method under different $q$. In \eqref{mk}, we employ $\alpha_{k}^{BB1}$ in the denominator to characterize the current curvature of problem. That is, if the  $\alpha_{k}^{BB1}$ is large, then $m_{k}$ tends to approach $0$, favoring $\alpha_{k}^{PBB}$ approximating $\alpha_{k}^{BB2}$. Otherwise, $m_{k}$ tends to approach $1$, thus causing $\alpha_{k}^{PBB}$ to approximate $\alpha_{k}^{BB1}$.


\begin{remark}
	When the ``situation" of the current iterate is ``bad" (search direction deviates from one eigenvector of Hessian or curvature is large), a short step size is generated by using interpolation parameter to improve the performance of algorithm. This strategy, which ``penalizes" the ``bad situation'' through adjustment of the interpolation parameter $m_{k}$, endows the PBB method with certain regularization characteristics.
\end{remark}

\section{\text{Nonquadratic} minimization}\label{sec:nonqua}
To extend \text{PBB} method for minimizing \text{nonquadratic} continuous differentiable functions \eqref{generalqua}, we usually need to incorporate some line search to ensure global convergence towards a local \text{minima}. When $f$ is a generic function, the average Hessian $A_{k}=\int_{0}^{1}\nabla^{2}f(\x_{k-1}+t\sss_{k-1})dt$ satisfies the secant equation $\yy_{k-1}=A_{k}\sss_{k-1}$ (cf. e.g., \cite[Eq.(6.11)]{JorgeNocedal2006NumericalOptimization}). Thus, we can still regard $\alpha_{k}^{BB1}$ and $\alpha_{k}^{BB2}$ as the Rayleigh quotients of $\g_{k-1}$ with respect to $A_{k}$ and $A_{k}^2$, respectively. We note that under the condition that $A_{k}$ is a symmetric positive definite (SPD) matrix, all results in Sects \ref{sec:2}, \ref{sec:convergence}, and \ref{sec:ABB} are still valid for generic functions with replacing $A$ by $A_{k}$.  

Among BB-like methods, the \text{non-monotonic} line search is regarded as an effective strategy \cite{Raydan1997BarzilaiBorweinGradienta}. Here we would like to adopt the \text{Grippo-Lampariello-Lucidi (GLL) non-monotonic} line search \cite{Grippo1986NonmonotoneLineSearch}, which accepts $\gamma_{k}\in(0,1)$ when it satisfies
\begin{equation}\label{non-monotone}
f(\x_{k}+\gamma_{k}\dd_{k})\le\max_{1\le j\le \min\{k,M\}} \{f(\x_{k-j+1})\} +\sigma\gamma_{k}\g_{k}^{\T}\dd_{k},	
\end{equation}  
where $M$ is a nonnegative integer, $\sigma\in(0,\,1)$, and  $\dd_{k}=-\frac{1}{\alpha_{k}}\g_{k}$. The detailed procedure is presented in Algorithm \ref{alg:VBB}. Line $2$ describes a condition for the function value to decrease sufficiently, where the common values of line search parameters $\delta=\frac{1}{2}$, $\sigma=10^{-4}$ (cf. \cite[p.33]{JorgeNocedal2006NumericalOptimization}), and $M=10$ in the experiments. At the start of each internal line search, we take $\gamma_{k}=1$. The critical assumption to prove the global convergence of Algorithm \ref{alg:VBB} is that step size $\frac{1}{\alpha_{k}}$ is uniformly bounded, i.e., $\alpha_{k}\in[\alpha_{\min}, \alpha_{\max}]$ for all $k>1$. Since the step size $\frac{1}{\alpha_{k}^{PBB}}$ \eqref{VBB} with safeguard lies in this interval, the convergence of Algorithm \ref{alg:VBB} is guaranteed by \cite[Thm.2.1]{Raydan1997BarzilaiBorweinGradient}. The $R$-linear convergence of Algorithm \ref{alg:VBB} can be proved for uniformly convex objective functions \cite{Dai2002NonmonotoneLineSearch}.
 
\begin{algorithm}[H]
	\SetAlgoLined
	\KwData{$\x_{1}$, $\varepsilon>0$, ${\text{MaxIt}>1}$, $\alpha_{1}\in [\alpha_{\text{min}},\alpha_{\text{max}}]$, $\sigma$, $\delta \in(0,1)$, $M$, $q$, $k=1$. }
	\KwResult{$\x_{k+1}$.}
	
	\While{$\|\g_{k}\|_{2}>\varepsilon$ \rm{or} $k<\rm{MaxIt}$}{
		\eIf{$f(\x_{k}-\gamma_{k}\frac{1}{\alpha_{k}}\g_{k})\le\max_{1\le j\le \min\{k,\,M\}} \{f(\x_{k-j+1})\} -\sigma\gamma_{k}\frac{1}{\alpha_{k}}\g_{k}^{\T}\g_{k}$}{
		  $\x_{k+1}=\x_{k}-\frac{1}{\alpha_{k}}\g_{k}$;\\
		  	\eIf{$\sss_{k}^{\T}\yy_{k}<0$}{$\alpha_{k+1}=\hat{\alpha}_{k+1}$;}{calculate $\bar{\alpha}_{k+1}=\alpha_{k+1}^{PBB}$ by \eqref{VBB} with \eqref{mk} \big($\bar{\alpha}_{k+1}=\alpha_{k+1}^{BB2}$ when $m_{k}=0$ in \eqref{mk}\big);\\
		  	set $\alpha_{k+1}=\min\{\max\{\bar{\alpha}_{k+1}, \alpha_{\min}\}, \alpha_{\max}\}$;}
			 set $k=k+1$;
		}{
			$\gamma_{k}=\delta\gamma_{k}$.
		}
	
	}
	\caption{Parameterized \BB method for nonquadratic problems}
	\label{alg:VBB}
\end{algorithm}
The selection of the initial step size and the treatment of uphill direction (i.e., $\sss_{k-1}^{\T}\yy_{k-1}\le0$) are two important factors affecting the performance of Algorithm \ref{alg:VBB}. Popular choice for the initial step size are $\frac{1}{\alpha_{1}}=1$ (cf. eg., \cite{Serafino2018steplengthselectiongradient,Raydan1997BarzilaiBorweinGradienta}) or $\frac{1}{\alpha_{1}}=\frac{1}{\|\g_{1}\|}$ (cf. eg., \cite{Dai2003Alternatestepgradient}), where the norm is the Euclidean $2$-norm or the $\infty$-norm. If $\sss_{k-1}^{\T}\yy_{k-1}<0$, then  $\alpha_{k}^{PBB}$ may be negative (except $m_{k}=0.5$). Therefore, the tentative $\alpha_{k}$ is replaced by a certain $\hat{\alpha}_{k}>0$. A possible choice is $\frac{1}{\hat{\alpha}_{k}}=\frac{\|\sss_{k-1}\|_2}{\|\yy_{k-1}\|_{2}}$ in  \cite{Burdakov2019StabilizedBarzilaiBorwein}. \text{Raydan} \cite{Raydan1997BarzilaiBorweinGradient} suggests using $\frac{1}{\hat{\alpha}_{k}}=\max(\min(\|\g_{k}\|_{2}^{-1}, 10^{5}),1)$, which makes the sequence $\{\frac{1}{\alpha_{k}}\}$ remains uniformly bounded while keeping $\|\frac{1}{\hat{\alpha}_{k}}\g_{k}\|_{2}$ moderate. Some authors use $\frac{1}{\hat{\alpha}_{k}}=\|\g_{k}\|_{2}^{-1}$, like the initial step size setting, similar to the restart operation.  There are other options, such as reuse the previous step size \cite{Park2020VariableMetricProximal}, which employs a cyclic gradient step size strategy \cite{Dai2006cyclicBarzilaiBorwein}. In Algorithm \ref{alg:VBB}, Line $7$, if $m_{k}<10^{-8}$, we set $\alpha_{k}^{PBB}=\alpha_{k}^{BB2}$ in practice, which is equivalent to a truncation strategy.    

\section{Numerical experiments}\label{sec:numerical} 
In this section, we conduct preliminary numerical experiments\footnote{All experiments were implemented in \text{MATLAB R2024a}. All the runs were carried out on a PC with an 12th Gen Intel(R) Core(TM) i7-12700H 2.30 GHz and 32 GB of RAM.} to illustrate the performance and convergence behavior of \text{PBB} method and numerically verify the validity of the parameter $m_{k}$ in \eqref{mk}. We first report the numerical performance of our step size $\frac{1}{\alpha_{k}^{PBB}}$ \eqref{VBB} with various $q$ in \eqref{mk}, then conduct numerical comparison experiments with some outstanding representative BB-like methods.  
\subsection{Choice of $q$}\label{subsec:choice q}
In this subsection, we evaluate the influence of the scaling factor $q$ on the interpolation parameter $m_{k}$ in \eqref{mk} using the performance profile \cite{Dolan2002Benchmarkingoptimizationsoftware}. The cost of solving each problem is normalized according to the lowest cost of solving that problem to obtain the performance ratio $\tau$. The most efficient method solves a given problem with a performance ratio $1$, while all other methods solve the problem with a performance ratio of at least $1$. In order to grasp the full implications of our test data regarding the solvers' probability of successfully handing a problem, we display a log scale of the performance profiles. Since we are also interested in the behavior for $\tau$ close to $1$, we use a base of $2$ for the scale \cite{Dolan2002Benchmarkingoptimizationsoftware}. Therefore, the value of $\rho_{s}(\log_2(1))$ is the probability that a solver $s$ will win over the rest of the solvers. Unless otherwise specified, the performance profiles mentioned in subsequent experiments are all $\log_2$ scaled. For convenience, we denote $\omega=\log_2(\tau)$ in this paper.
 
Consider the following quadratic function from \cite{DeAsmundis2014efficientgradientmethod}: 
\begin{equation}\label{pro:qua}
f(\x)=\frac{1}{2}(\x-\x_{*})^{\T}A(\x-\x_{*}),
\end{equation}
where $\x_{*}$ is uniformly and randomly generated from $[-10, 10]^{n}$, $A=Q\cdot\text{diag}(v_{1},\ldots,v_{n})\cdot Q^{\T}$ with  $Q=(I-2\omega_3\omega_3^{\T})(I-2\omega_2\omega_2^{\T})(I-2\omega_1\omega_1^{\T})$ for $\omega_1$, $\omega_2$, and $\omega_3$ being unit random vectors, $v_1=1$, $v_n=\kappa$, and $v_j$ is randomly generated between $1$ and $\kappa$ for $j=2,\ldots, n-1$. 

We set $n\in\{100, 1000\}$ and $\kappa=10^3, 10^4, 10^5, 10^6$. The initial point is the null vector of $\mathbb{R}^{n}$. The stopping criterion is either that the gradient at the $k$th iteration satisfies that $\|\g_{k}\|_{2}\le\varepsilon\|\g_{1}\|_{2}$ with $\varepsilon=10^{-6}, 10^{-8}, 10^{-10}$ or the number of iterations exceeds $20000$. The initial step size is $\frac{\g_{1}^{\T}\g_{1}}{\g_{1}^{\T}A\g_{1}}$. For each of the above settings, we execute the program independently $10$ times to examine the impact of different $q$-values on the performance of $m_{k}$. We consider the case where $q$ gradually increases from $1$ to $8$. Intuitively, a larger value of $q$ facilitates the rapid approximate of $m_{k}$ towards either $1$ or $0$. Figure \ref{fig:parameterVBB} presents the performance profiles by the PBB methods with these $q$-values using the number of iterations as the metric. 

It is evident from Figure \ref{fig:parameterVBB} that the best choice of $q$ appear to be $4$ or $8$ for $n=100, 1000$. Through this paper, based on this result, we set $q=8$ in $\alpha_{k}^{VBB}(m_{k})$. In this problem, compared with BB1 and BB2, \text{PBB} is effective even for low $q$. This demonstrates that the generation scheme \eqref{mk} for parameter $m_{k}$ is effective.

\begin{figure}[h!]
	\centering
	\subfigure[n=100]{
		\includegraphics[width=0.5\textwidth]{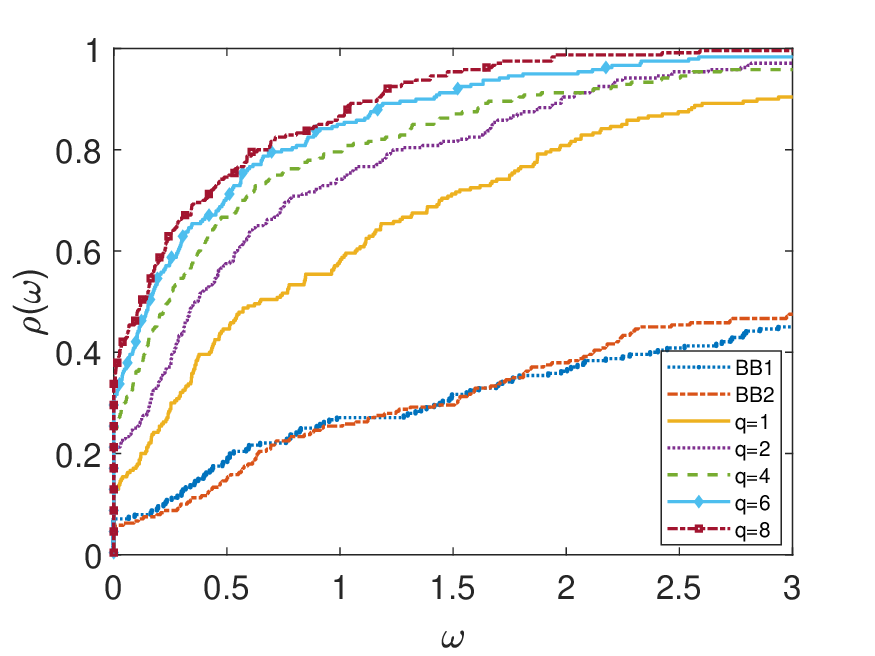}}\hspace{-15pt}
	\subfigure[n=1000]{
		\includegraphics[width=0.5\textwidth]{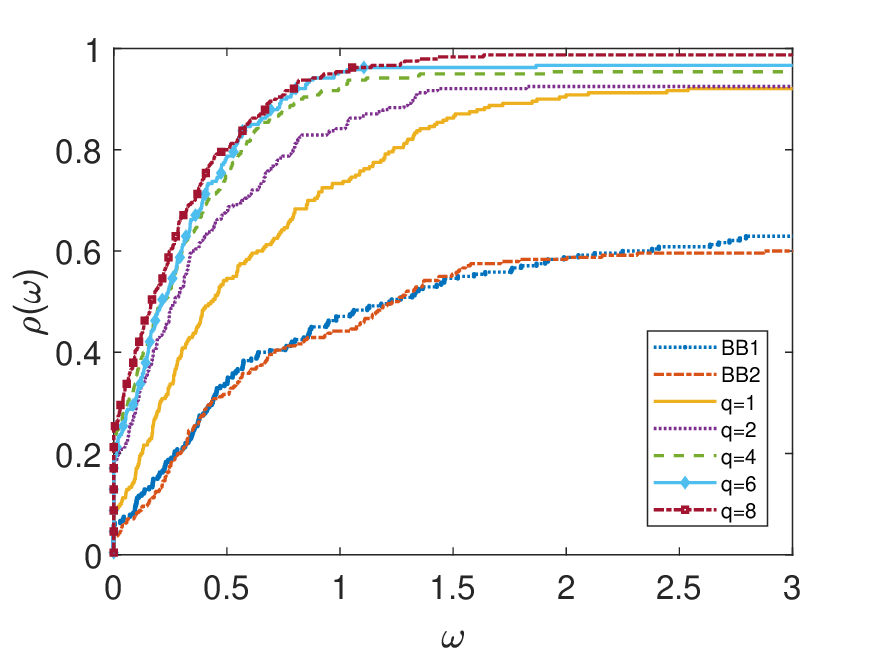}}\\
	\caption{\textit{Performance profiles of $\alpha_{k}^{PBB}(m_{k})$ with different $q$ for $n=100$, $n=1000$, number of iterations.}}	
	\label{fig:parameterVBB}	
\end{figure}

\subsection{Comparing with the state-of-the-art algorithms for quadratic problems}
We consider the quadratic function \eqref{pro:qua} with seven kinds of settings of $v_{j}$ summarized in Table \ref{tab:spectrum} from \cite{Dai2019familyspectralgradient}. We set $n=1000$ and the other settings including the stop criterion are the same as in Section \ref{subsec:choice q}. To evaluate the performance of the PBB method, we compare it with several representative BB-like methods. \cite{Dai2019familyspectralgradient} proposes an adaptive truncated cyclic (ATC) strategy, and numerical results show that it is an effective BB-like step size:
\begin{equation}\label{ATC}
	\alpha_{k}^{ATC}=\begin{cases}
	\alpha_{k}^{BB1},\quad&\text{if}\quad\text{mod}(k,m)=0,\\
	\widetilde{\alpha}_{k},\quad&\text{otherwise},
	\end{cases}
\end{equation} 
where the cyclic length $m$ is a positive integer and 
\begin{equation*}
\widetilde{\alpha}_{k}=\begin{cases}
\alpha_{k}^{BB1},\quad&\text{if}\quad\alpha_{k-1}\le\alpha_{k}^{BB1},\\
\alpha_{k}^{BB2},\quad&\text{if}\quad\alpha_{k-1}\ge\alpha_{k}^{BB2},\\
\alpha_{k-1},\quad&\text{otherwise}.
\end{cases}
\end{equation*} 
\cite{Ferrandi2023harmonicframeworkstepsize} provides a harmonic framework step size as follows
\begin{equation}\label{TBB}
\alpha_{k}^{TBB}(\tau_{k})=\frac{\yy_{k-1}^{\T}(\yy_{k-1}-\tau_{k}\sss_{k-1})}{\sss_{k-1}^{\T}(\yy_{k-1}-\tau_{k}\sss_{k-1})},
\end{equation}
where $\tau_{k}$ is the target parameter. In a sense, $\alpha_{k}^{TBB}$ is equivalent to a generalized convex combination of $\alpha_{k}^{BB1}$ and $\alpha_{k}^{BB2}$, which is adjusted by the target parameter $\tau_{k}$.

\begin{table}[h!]
	\centering
	\caption{Different settings of $v_{j}$ for the problem \eqref{pro:qua}.}
	\setlength{\tabcolsep}{10pt}{
		\begin{tabular}{cccc|cccc}
		\toprule
		P     & \multicolumn{3}{c|}{$v_{j}$} & P     & \multicolumn{3}{c}{$v_{j}$} \\
		\midrule
		1     & \multicolumn{3}{c|}{$\{v_{2},\ldots,v_{n-1}\}\subset(1,\kappa)$} & \multirow{3}[4]{*}{5} & \multicolumn{3}{c}{$\{v_{2},\ldots,v_{n/5}\}\subset(1,100)$} \\
		\cmidrule{1-4}    \multirow{2}[2]{*}{2} & \multicolumn{3}{c|}{$\{v_{2},\ldots,v_{n/5}\}\subset(1,100)$} &       & \multicolumn{3}{c}{$\{v_{n/5+1},\ldots,v_{4n/5}\}\subset(100,\frac{\kappa}{2})$} \\
		& \multicolumn{3}{c|}{$\{v_{n/5+1},\ldots,v_{n-1}\}\subset(\frac{\kappa}{2},\kappa)$} &       & \multicolumn{3}{c}{$\{v_{4n/5+1},\ldots,v_{n-1}\}\subset(\frac{\kappa}{2}, \kappa)$} \\
		\midrule
		\multirow{2}[2]{*}{3} & \multicolumn{3}{c|}{$\{v_{2},\ldots,v_{\frac{n}{2}}\}\subset(1,100)$} & \multirow{2}[2]{*}{6} & \multicolumn{3}{c}{$\{v_{2},\ldots,v_{10}\}\subset(1,100)$} \\
		& \multicolumn{3}{c|}{$\{v_{n/2+1},\ldots,v_{n-1}\}\subset(\frac{\kappa}{2},\kappa)$} &       & \multicolumn{3}{c}{$\{v_{11},\ldots,v_{n-1}\}\subset(\frac{\kappa}{2},\kappa)$} \\
		\midrule
		\multirow{2}[2]{*}{4} & \multicolumn{3}{c|}{$\{v_{2},\ldots,v_{4n/5}\}\subset(1,100)$} & \multirow{2}[2]{*}{7} & \multicolumn{3}{c}{$\{v_{2},\ldots,v_{n-10}\}\subset(1,100)$} \\
		& \multicolumn{3}{c|}{$\{v_{4n/5+1},\ldots,v_{n-1}\}\subset(\frac{\kappa}{2},\kappa)$} &       & \multicolumn{3}{c}{$\{v_{n-9},\ldots,v_{n-1}\}\subset(\frac{\kappa}{2},\kappa)$} \\
		\bottomrule
	\end{tabular}}%
	\label{tab:spectrum}%
\end{table}%

We compare the \text{PBB} method with the BB1, BB2 \eqref{BB steps}, \text{ABB} \eqref{ABB}, \text{ATC} \eqref{ATC}, and \text{TBB} \eqref{TBB} methods, where the parameters of \text{ABB} and \text{ATC} are the same as those in \cite{Dai2019familyspectralgradient}, and the parameter of \text{TBB} is taken as $\tau_{k}=-\cot\theta_{k}$ as suggested in  \cite{Ferrandi2023harmonicframeworkstepsize}. In addition, we compare it with the BBQ algorithm \cite{Huang2021EquippingBarzilaiBorwein} which has the two dimensional quadratic termination property, and the parameters are selected according to the suggestion of \cite{Huang2021EquippingBarzilaiBorwein}. We also consider two effective variants of \text{ABB}, which we indicate with \text{ABBmin} \cite{Frassoldati2008Newadaptivestepsize} and \text{ABBbon} \cite{Bonettini2009scaledgradientprojection}. In the first one, we select the smallest BB2 step size over the last $m+1$ iterations when $\cos^{2}\theta_{k}$ is small:
\begin{equation}
	\alpha_{k}^{ABBmin}=\begin{cases}
	\max\{\alpha_{j}^{BB2}|j=\max\{1,k-m\},\ldots,k\},\quad&\text{if}\quad\cos^2\theta_{k}<\xi,\\
	\alpha_{k}^{BB1},\quad&\text{otherwise}.
	\end{cases}
\end{equation} 
\text{ABBbon} is defined in the same way as \text{ABBmin} but with an adaptive threshold value $\xi_{k}$ as follows
\begin{equation}
\xi_{k+1}=\begin{cases}
0.9\xi_{k},\quad&\text{if}\quad\cos^2\theta_{k}<\xi_{k},\\
1.1\xi_{k},\quad&\text{otherwise},
\end{cases}
\end{equation} 
with $\xi_{0}=0.5$. It should be noted that ABB, ABBmin, and ABBbon are currently recognized as efficient spectral gradient descent methods. BBQ demonstrates exceptional performance, leveraging its two dimensional quadratic termination property. Both ATC and TBB can be viewed as certain combination of BB1 and BB2. Consequently, we choose these algorithms for comparison with PBB to validate the latter's effectiveness.

Figure \ref{fig:QuaCompare} displays the performance profiles by these different gradient descent methods using the number of iterations as the metric. We focus on observing the height of $\rho(\omega)$ corresponding to different methods at $\omega=0$. It is evident from Figure \ref{fig:QuaCompare} that BBQ and \text{ABBbon} perform best overall among these compared methods, which is attributed to the fact that BBQ has the two dimensional quadratic termination property on one hand, and both of them use the adaptive alternating threshold strategy on the other hand. Under low-precision conditions (i.e., $\varepsilon=10^{-6}, 10^{-8}$), \text{PBB} demonstrates significant superiority over \text{ABBmin}, ABB, and the other compared algorithms. Under high-precision condition (i.e., $\varepsilon=10^{-10}$), \text{PBB} performs comparably to ABBmin while outperforming ABB and the other compared algorithms. As analyzed in \cite{Ferrandi2023harmonicframeworkstepsize}, \text{ABBbon} and \text{ABBmin} are particularly beneficial for quadratic problems: where the Hessian matrix is constant, the old $\alpha_{j}^{BB2}$ for some $j<k$ still approximates some eigenvalues of the Hessian matrix. For nonquadratic problems, the effectiveness of this delay operation deteriorates, and the quadratic termination property of BBQ also degrades. Specifically, see the numerical results in Section \ref{sec:NonQua}.  
\begin{figure}[h!]
	\centering
	\subfigure[$\varepsilon=10^{-6}$]{
		\includegraphics[width=0.34\textwidth]{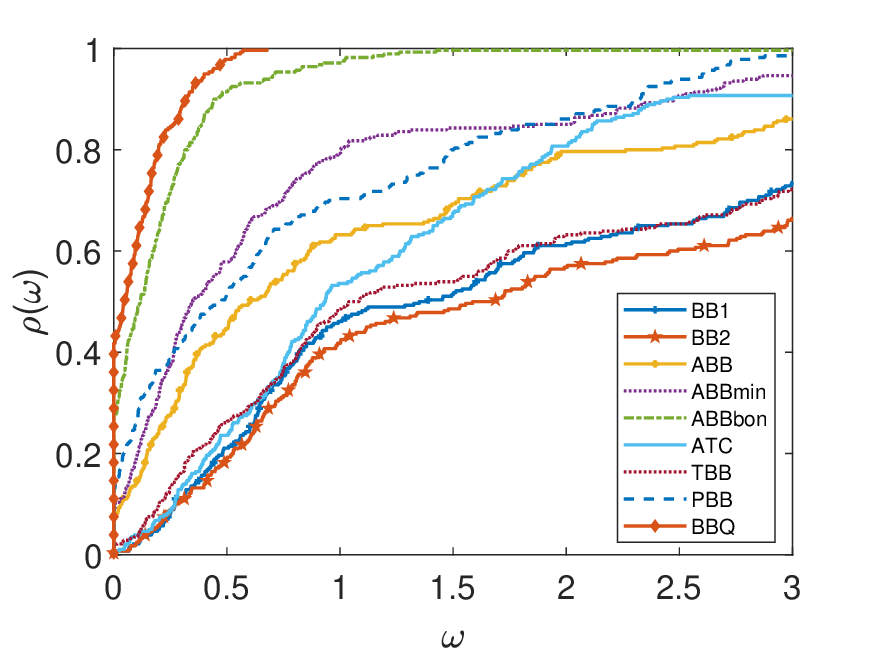}}\hspace{-15pt}
	\subfigure[$\varepsilon=10^{-8}$]{
		\includegraphics[width=0.34\textwidth]{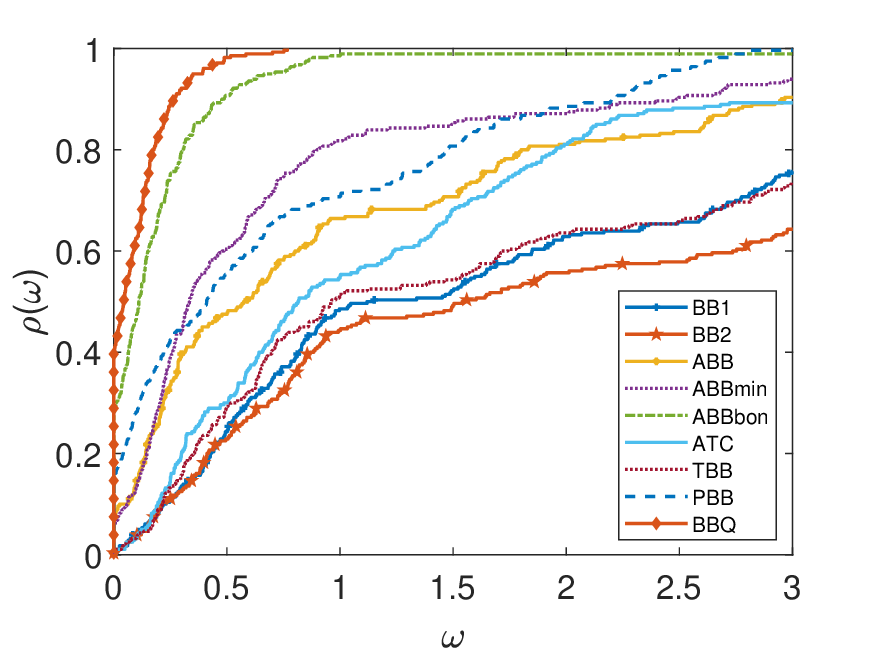}}\hspace{-15pt}
	\subfigure[$\varepsilon=10^{-10}$]{
		\includegraphics[width=0.34\textwidth]{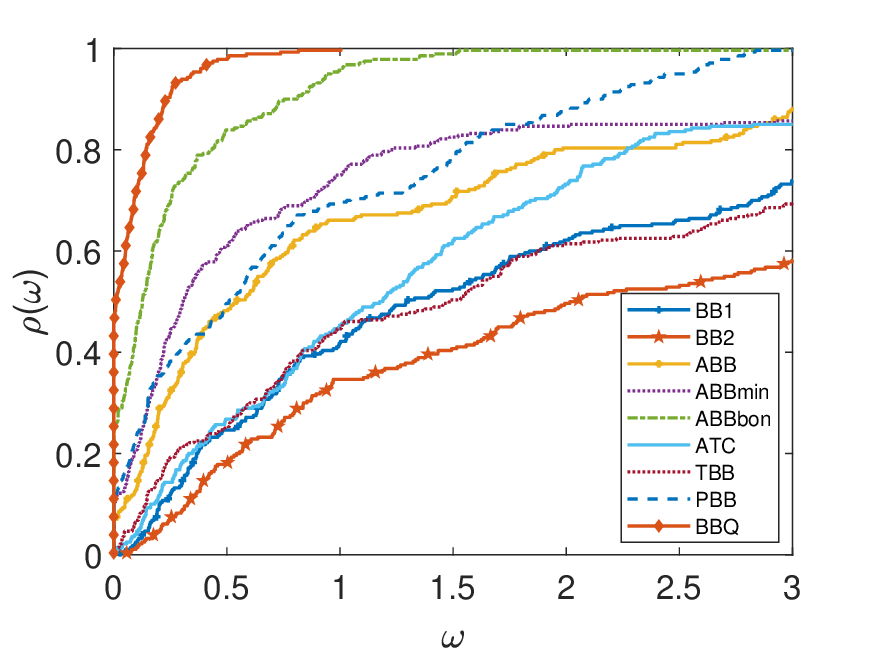}}\\
	\caption{\textit{Performance profiles of  different spectral gradient descent methods on random quadratic problem \eqref{pro:qua} with spectral distributions in Table \ref{tab:spectrum}, iteration numbers metric.}}	
	\label{fig:QuaCompare}	
\end{figure}

We next compare these methods on a two-point boundary value problem \cite{Dai2003Alternateminimizationgradient,Huang2022accelerationBarzilaiBorweinmethod} which can be transferred as a linear system $A\x=\mathbf{b}$ by the finite difference method. In particular, the matrix $A=(a_{i,j})$ is given by
\begin{equation}\label{A}
	a_{i,j}=\begin{cases}
	\frac{2}{h^2},\quad&\text{if}\quad i=j,\\
	-\frac{1}{h^2},\quad&\text{if}\quad i=j\pm 1,\\
	0,\quad&\text{otherwise},
	\end{cases}
\end{equation}
where $h=11/n$. Obviously, the condition number $\kappa$ of matrix $A$ increases with its size $n$. Likewise, we consider the objective function \eqref{pro:qua}, where $\x_{*}$ is a randomly generated as the preceding test. 

We set the cycle depth $m=8$ in \text{ATC} algorithm, and the parameter settings of the other compared algorithms are consistent with those in the preceding test. We consider $n=500, 1000, 2000, 3000, 5000$, with the initial point $\x_{1}=\mathbf{1}$. Ten independent runs are performed with these settings. Figure \ref{fig:Two-pointValue} presents the performance profiles by PBB and other compared methods using the number of iterations as the metric.

As evidenced in Figure \ref{fig:Two-pointValue}, the performance of PBB is comparable to that of BBQ at convergence tolerances of $\varepsilon=10^{-4}$ and $\varepsilon=10^{-8}$, with both methods outperforming the remaining approaches. When the convergence tolerance $\varepsilon=10^{-6}$, PBB demonstrates superior performance, significantly outperforming BBQ as well as other compared methods.
\begin{figure}[h!]
	\centering
	\subfigure[$\varepsilon=10^{-4}$]{
		\includegraphics[width=0.34\textwidth]{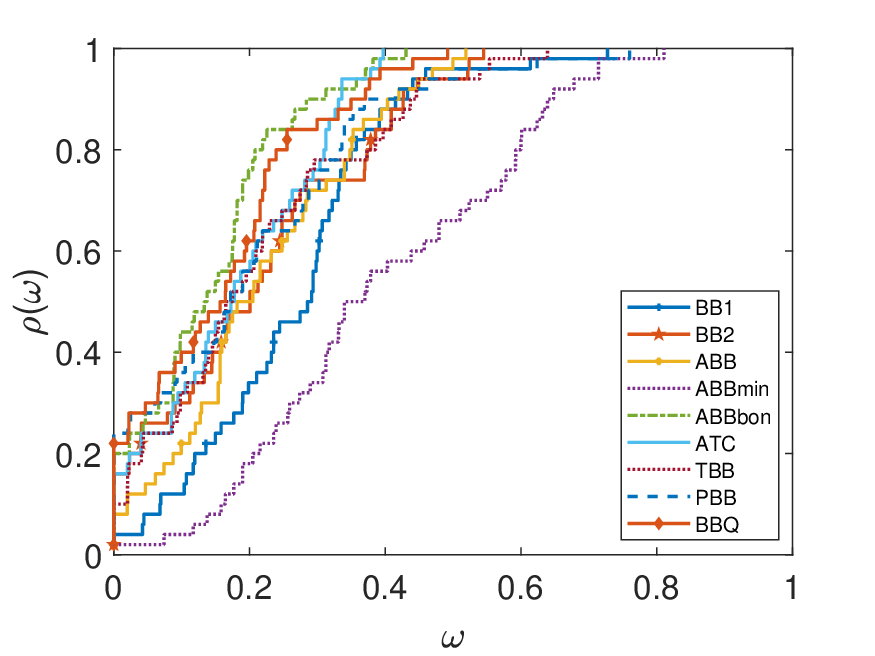}}\hspace{-15pt}
	\subfigure[$\varepsilon=10^{-6}$]{
		\includegraphics[width=0.34\textwidth]{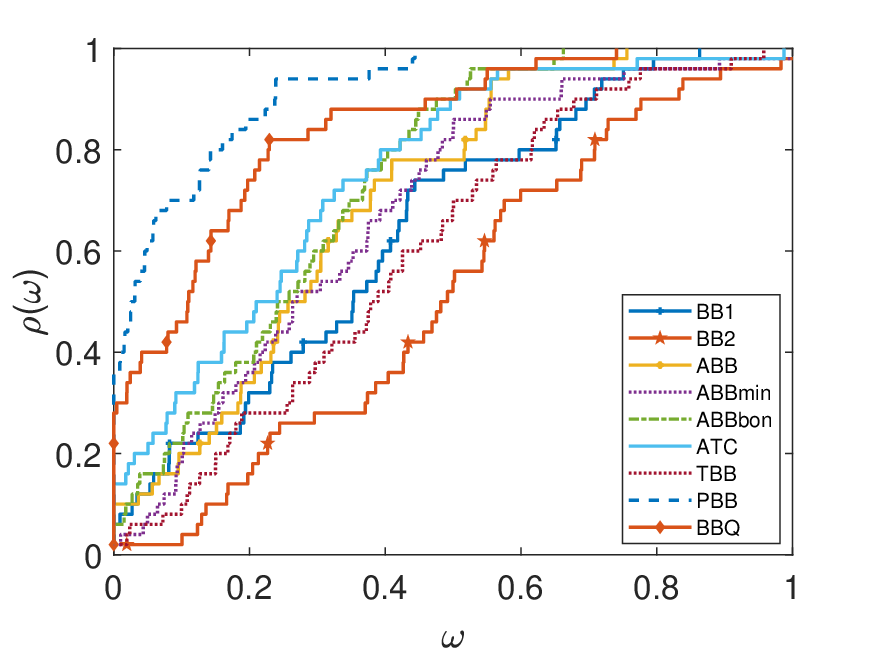}}\hspace{-15pt}
	\subfigure[$\varepsilon=10^{-8}$]{
		\includegraphics[width=0.34\textwidth]{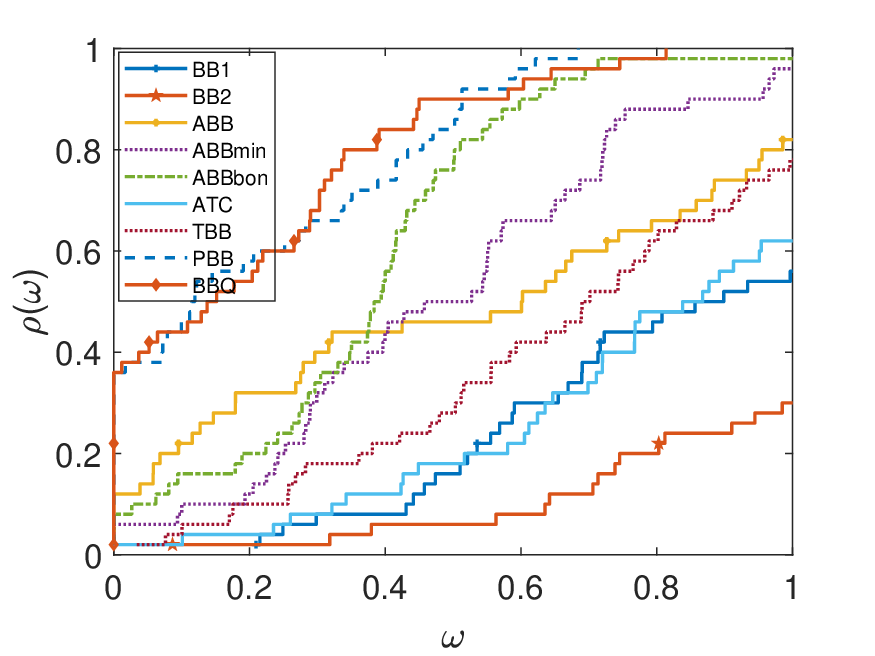}}\\
	\caption{\textit{Performance profiles of  different spectral gradient descent methods on the two-point boundary value problem with $A$ from \eqref{A}, iteration numbers metric.}}	
	\label{fig:Two-pointValue}	
\end{figure}

\subsection{Solving nonquadratic minimization problems}\label{sec:NonQua}
In this part, we solve some \text{nonquadratic} optimization problems. Generally, as described in Section \ref{sec:nonqua}, we need more adjustable parameters than in quadratic optimization. We maintain the parameter settings in \cite{Serafino2018TwoPhaseGradient}, set $\alpha_{\min}=10^{-30}$, $\alpha_{\max}=10^{30}$, $\sigma=10^{-4}$, $\delta=\frac{1}{2}$, $M=10$, and the maximum number of internal non-monotonic line searches to $100$ per external iteration. The wide step size bound is to choose BB-like step size as much as possible. Following Raydan \cite{Raydan1997BarzilaiBorweinGradienta}, we choose $\frac{1}{\hat{\alpha}_{k}}=\max(\min(\|\g_{k}\|_{2}^{-1}, 10^{5}),1)$ as the replacement for negative step size, and take $\frac{1}{\alpha_{1}}=1$ as initial step size. We terminate the algorithm when $\|\g_{k}\|_{2}<\varepsilon\|\g_{1}\|_{2}$ or the number of iterations reaches $20000$ or the number of function evaluations reaches $10^5$. The parameters of each algorithm are consistent with the preceding settings. 

In terms of comparing the performance of different algorithms, we consider the number of function evaluations rather than the number of iterations, which is equal to the number of gradient evaluations. In algorithms with line search operation, each execution of line search requires evaluating the function value but not the gradient. Therefore, the number of function evaluations is the main computational cost and includes the number of iterations.  

We first consider the classical Rosenbrock function \cite{Rosenbrock1960AutomaticMethodFinding}
\begin{equation}\label{equ:Rosen}
	f(\x)=c\big(\x^{(2)}-(\x^{(1)})^2\big)^{2}+\big(1-\x^{(1)}\big)^{2},	
\end{equation}
which is often used as a test case for optimization algorithms, where $c$ is a constant that controls the difficulty of the problem, the initial point is $\big(\x_{1}^{(1)}, \x_{1}^{(2)}\big)=(-1.2, 1)$. The global minimum is inside a long, narrow, and parabolic-shaped flat valley. Setting the initial step size $\alpha_{1}^{-1}=1$, we compare the performance of these spectral gradient descent methods. We set the stop criterion as $\|\big(\x_{k}^{(1)}, \x_{k}^{(2)}\big)-\big(\x_{*}^{(1)}, \x_{*}^{(2)}\big)\|_{2}<\varepsilon$, where $\big(\x_{*}^{(1)}, \x_{*}^{(2)}\big)=(1, 1)$ is the minimizer of $f(\x)$ in \eqref{equ:Rosen}. We report in Table \ref{tab:Rosenbrock} the number of function evaluations with different $\varepsilon$. 
\begin{center}
	\setlength{\tabcolsep}{9pt}{
		\begin{longtable}[h!]{ccccccccccc}
			\captionsetup{width=0.9\textwidth}
			\caption{The performance of the BB1, BB2, ABB, ABBmin, ABBbon, ATC, TBB, BBQ, PBB methods on the Rosenbrock problem \eqref{equ:Rosen}, based on the number of function evaluations.}
			\label{tab:Rosenbrock}\\
			\toprule
			\multicolumn{1}{c}{P} & \multicolumn{1}{c}{$\varepsilon$} & \multicolumn{1}{c}{BB1}&
			\multicolumn{1}{c}{BB2}&
			\multicolumn{1}{c}{ABB}&
			\multicolumn{1}{c}{ABBmin}&
			\multicolumn{1}{c}{ABBbon}&
			\multicolumn{1}{c}{ATC}&
			\multicolumn{1}{c}{TBB}&
			\multicolumn{1}{c}{BBQ}&
			\multicolumn{1}{c}{PBB}\\ 
			\midrule
			
			\endfirsthead
			
			\multicolumn{10}{c}%
			{{\bfseries \tablename\ \thetable{} -- continued from previous page}} \\
			\toprule \multicolumn{1}{c}{c} & \multicolumn{1}{c}{$\varepsilon$} & \multicolumn{1}{c}{BB1}&
			\multicolumn{1}{c}{BB2}&
			\multicolumn{1}{c}{ABB}&
			\multicolumn{1}{c}{ABBmin}&
			\multicolumn{1}{c}{ABBbon}&
			\multicolumn{1}{c}{ATC}&
			\multicolumn{1}{c}{TBB}&
			\multicolumn{1}{c}{BBQ}&
			\multicolumn{1}{c}{PBB}\\ 
			\midrule 
			\endhead
			
			\hline \multicolumn{10}{l}{{Continued on next page}} \\ 
			\endfoot
			
			\bottomrule
			\endlastfoot
		\multirow{4}[2]{*}{$10^2$} & $10^{-1}$ & 92    & 68    & 105   & 107   & 146   & 99    & 2252  & -     & \textbf{67} \\
		& $10^{-2}$ & 100   & 75    & 110   & 197   & 224   & 109   & -     & -     & \textbf{73} \\
		& $10^{-4}$ & 107   & 81    & 110   & 665   & 640   & 116   & -     & -     & \textbf{79} \\
		& $10^{-8}$ & 115   & 89    & 119   & 667   & 655   & 122   & -     & -     & \textbf{85} \\
		\midrule
		\multirow{4}[2]{*}{$10^3$} & $10^{-1}$ & \textbf{184} & 190   & 309   & 355   & 355   & 313   & -     & -     & 214 \\
		& $10^{-2}$ & 195   & \textbf{190} & 344   & 439   & 441   & 345   & -     & -     & 220 \\
		& $10^{-4}$ & 207   & \textbf{197} & 355   & 672   & 689   & 363   & -     & -     & 227 \\
		& $10^{-8}$ & 212   & \textbf{203} & 361   & 1156  & 704   & 370   & -     & -     & 233 \\
		\midrule
		\multirow{4}[2]{*}{$10^{4}$} & $10^{-1}$ & 548   & 475   & 593   & 651   & \textbf{347} & 708   & -     & -     & 485 \\
		& $10^{-2}$ & 571   & 510   & 627   & 732   & \textbf{428} & 791   & -     & -     & 508 \\
		& $10^{-4}$ & 587   & 517   & 646   & 930   & 750   & 804   & -     & -     & \textbf{515} \\
		& $10^{-8}$ & 595   & 606   & 653   & 2169  & 2180  & 836   &-     & -     & \textbf{531} \\
		\midrule
		\multirow{4}[2]{*}{$10^5$} & $10^{-1}$ & 1685  & \textbf{844} & 1035  & 1621  & 1622  & 1474  & -     & -     & 970 \\
		& $10^{-2}$ & 1790  & \textbf{910} & 1122  & 1723  & 1726  & 1624  & -     & -     & 1033 \\
		& $10^{-4}$ & 1813  & \textbf{910} & 1122  & 2024  & 2028  & 1660  & -     & -     & 1038 \\
		& $10^{-8}$ & 1827  & -     & 1134  & 3079  & 2390  & 1691  & -     & -     & \textbf{1045} \\	
\end{longtable}}
\end{center}

It is observed from the numerical results in Table \ref{tab:Rosenbrock} that \text{PBB} performs best when $c=10^2$, \text{BB1} and \text{BB2} have advantages for $c=10^3$, when $c=10^4, 10^5$, \text{ABBbon} and \text{BB2} perform better than \text{PBB} for low-precision requirements, but \text{PBB} performs well for high-precision requirements, respectively. It can also be observed that \text{PBB} is comparable with \text{BB2} in this \text{nonquadratic} test problem, \text{ABBbon} and \text{ABBmin} no longer have a dominant advantage. In the table, ``-'' indicates that the number of function evaluations exceeds $40000$. One possible reason for this phenomenon is that simply using the cotangent function cannot make $\alpha_{k}^{TBB}$ quickly approach $\alpha_{k}^{BB1}$ or $\alpha_{k}^{BB2}$, causing the iterate to fall into a narrow valley, resulting in a slow convergence speed. The quadratic truncation property of the BBQ method is no longer valid in this nonquadratic problem. Similar numerical phenomena can readily occur in other nonquadratic problems. Therefore, we do not report the \text{TBB} and BBQ methods in the following comparative experiments.    

We then conduct numerical comparison experiments in some commonly used test functions from  \cite{More1981TestingUnconstrainedOptimization,Andrei2008UnconstrainedOptimizationTest,Dixon1989TruncatedNewtonmethod} and the suggested starting points $\x_{1}$ therein, where the names and scales of the tested functions are listed in Table \ref{tab:test set}. The parameter settings, initial point, and termination conditions are as described at the beginning of this subsection, where the tolerance is set to  $\varepsilon=10^{-4}, 10^{-6}, 10^{-8}$. 

\begin{table}[h!]
	\centering
	\caption{Test functions of nonquadratic minimization.}
	\label{tab:test set}
	\setlength{\tabcolsep}{9pt}{
		\begin{tabular}{ccc|ccc}
			\toprule
		P     & Name  & $n$     & P     & Name  & $n$ \\
		\midrule
		P1    & Almost Perturbed Quadratic & 100   & P18   & Extended Powell & 100 \\
		P2    & BIGGSB1 & 100   & P19   & Extended Rosenbrock & 50 \\
		P3    & CUBE  & 2     & P20   & Extended Beale & 100 \\
		P4    & Diagonal4 & 100   & P21   & Extended quadratic penalty QP2 & 100 \\
		P5    & Dixon Price & 100   & P22   & FLETCHCR & 50 \\
		P6    & DIXON3DQ & 100   & P23   & Generalized Rosenbrock & 10 \\
		P7    & DQDRTIC & 100   & P24   & HIMMELBG & 100 \\
		P8    & DIXMAANI & 100   & P25   & LIARWHD & 100 \\
		P9    & DIXMAANJ & 100   & P26   & MCCORMCK & 100 \\
		P10   & DIXMAANK & 100   & P27   & NONSCOMP & 100 \\
		P11   & DIXMAANL & 100   & P28   & NONDIA & 100 \\
		P12   & DIXMAANM & 100   & P29   & Perturbed Quadratic & 100 \\
		P13   & DIXMAANN & 100   & P30   & Perturbed QuadraticDiagonal & 100 \\
		P14   & DIXMAANP & 100   & P31   & Perturbed Tridiagonal Quadratic & 100 \\
		P15   & Extended DENSCHNF & 100   & P32   & POWER & 2000 \\
		P16   & Extended Himmelblau & 100   & P33   & Staircase1 & 100 \\
		P17   & Extended White Holst & 100   &       &       &  \\
	\bottomrule
\end{tabular}}%
\end{table}%

The numerical results are presented in Table \ref{tab:non-quadratic results}. We report the  number of function evaluations required by different methods and summarize the total number of evaluations at the bottom of Table \ref{tab:non-quadratic results}, for different tolerances. From the numerical results, we can see that \text{PBB} requires the smallest number of evaluations, followed by \text{ABB}, BB2, ATC, ABBbon, BB1, and ABBmin. For P12, ``-" is filled for the \text{ABBmin} method since it takes more than maximum number of function evaluations to satisfy the stopping condition. 

Based on the numerical results, it can be concluded that for nonquadratic optimization problems, employing old step size leads to performance degradation of the algorithm, since the local Hessian of nonquadratic objective function ceases to be constant. Furthermore, selecting shorter spectral gradient step sizes appears more suitable for general optimization problems.   
\begin{center}
	\setlength{\tabcolsep}{10pt}{
		\begin{longtable}[h!]{ccccccccc}
			\caption{The number of function evaluations required by the BB1, BB2, ABB, ABBmin, ABBbon, ATC, PBB methods, for the nonquadratic functions in Table \ref{tab:test set}.} \label{tab:non-quadratic results} \\
			\toprule
			\multicolumn{1}{c}{P} & \multicolumn{1}{c}{$\varepsilon$} & \multicolumn{1}{c}{BB1}&
			\multicolumn{1}{c}{BB2}&
			\multicolumn{1}{c}{ABB}&
			\multicolumn{1}{c}{ABBmin}&
			\multicolumn{1}{c}{ABBbon}&
			\multicolumn{1}{c}{ATC}&
			\multicolumn{1}{c}{PBB}\\ 
			\midrule
			
			\endfirsthead
			
			\multicolumn{9}{c}%
			{{\bfseries \tablename\ \thetable{} -- continued from previous page}} \\
			\toprule \multicolumn{1}{c}{P} & \multicolumn{1}{c}{$\varepsilon$} & \multicolumn{1}{c}{BB1}&
			\multicolumn{1}{c}{BB2}&
			\multicolumn{1}{c}{ABB}&
			\multicolumn{1}{c}{ABBmin}&
			\multicolumn{1}{c}{ABBbon}&
			\multicolumn{1}{c}{ATC}&
			\multicolumn{1}{c}{PBB}\\ 
			\midrule 
			\endhead
			
			\hline \multicolumn{9}{l}{{Continued on next page}} \\ 
			\endfoot
			
			\bottomrule
			\endlastfoot
			
		\multirow{3}[1]{*}{P1} & $10^{-4}$ & 14    & 14    & 14    & 14    & 15    & 14    & 14 \\
		& $10^{-6}$ & 20    & 20    & 20    & 20    & 27    & 20    & 20 \\
		& $10^{-8}$ & 36    & 35    & 36    & 35    & 37    & 32    & 31 \\
		\multirow{3}[0]{*}{P2} & $10^{-4}$ & 545   & 398   & 281   & 179   & 254   & 375   & 237 \\
		& $10^{-6}$ & 1232  & 526   & 344   & 267   & 474   & 584   & 288 \\
		&$10^{-8}$ & 1684  & 631   & 603   & 391   & 644   & 1014  & 464 \\
		\multirow{3}[0]{*}{P3} & $10^{-4}$ & 14    & 14    & 14    & 14    & 14    & 14    & 14 \\
		& $10^{-6}$ & 90    & 44    & 80    & 37    & 37    & 88    & 33 \\
		&$10^{-8}$ & 95    & 50    & 95    & 41    & 41    & 96    & 38 \\
		\multirow{3}[0]{*}{P4} & $10^{-4}$ & 8     & 8     & 8     & 9     & 8     & 9     & 9 \\
		&  $10^{-6}$ & 11    & 8     & 11    & 9     & 10    & 11    & 11 \\
		&$10^{-8}$ & 11    & 11    & 11    & 9     & 10    & 11    & 11 \\
		\multirow{3}[0]{*}{P5} & $10^{-4}$ & 27    & 22    & 27    & 24    & 30    & 27    & 22 \\
		& $10^{-6}$ & 280   & 198   & 107   & 120   & 107   & 185   & 84 \\
		& $10^{-8}$ & 2849  & 1669  & 1695  & 7931  & 6673  & 1831  & 1510 \\
		\multirow{3}[0]{*}[1.2em]{P6} & $10^{-4}$ & 926   & 330   & 260   & 185   & 236   & 259   & 152 \\
		& $10^{-6}$& 1538  & 369   & 348   & 258   & 499   & 563   & 261 \\
		&$10^{-8}$ & 2081  & 630   & 577   & 372   & 755   & 918   & 410 \\
		\multirow{3}[0]{*}{P7} & $10^{-4}$ & 19    & 16    & 16    & 13    & 12    & 15    & 16 \\
		& $10^{-6}$ & 20    & 34    & 19    & 16    & 16    & 41    & 17 \\
		&$10^{-8}$ & 38    & 39    & 20    & 19    & 20    & 61    & 26 \\
		\multirow{3}[0]{*}{P8} & $10^{-4}$& 13    & 14    & 13    & 15    & 15    & 14    & 13 \\
		& $10^{-6}$ & 39    & 35    & 38    & 36    & 32    & 32    & 31 \\
		&$10^{-8}$ & 117   & 99    & 98    & 96    & 96    & 123   & 89 \\
		\multirow{3}[0]{*}{P9} & $10^{-4}$ & 14    & 15    & 14    & 19    & 17    & 15    & 15 \\
		& $10^{-6}$& 123   & 65    & 89    & 98    & 91    & 81    & 84 \\
		&$10^{-8}$ & 849   & 887   & 467   & 278   & 526   & 678   & 300 \\
		\multirow{3}[0]{*}{P10} & $10^{-4}$ & 20    & 19    & 20    & 37    & 27    & 21    & 20 \\
		& $10^{-6}$ & 67    & 69    & 61    & 98    & 272   & 81    & 72 \\
		&$10^{-8}$ & 395   & 411   & 318   & 872   & 1040  & 394   & 272 \\
		\multirow{3}[0]{*}{P11} & $10^{-4}$ & 11    & 11    & 11    & 16    & 14    & 11    & 11 \\
		& $10^{-6}$ & 35    & 42    & 51    & 50    & 41    & 41    & 40 \\
		&$10^{-8}$ & 1023  & 625   & 354   & 297   & 299   & 423   & 216 \\
		\multirow{3}[0]{*}{P12} & $10^{-4}$ & 63    & 54    & 50    & 72    & 61    & 60    & 47 \\
		& $10^{-6}$ & 853   & 377   & 426   & - & 292   & 368   & 300 \\
		& $10^{-8}$ & 1298  & 682   & 588   & - & 497   & 715   & 583 \\
		\multirow{3}[0]{*}{P13} & $10^{-4}$ & 26    & 29    & 26    & 35    & 32    & 27    & 26 \\
		& $10^{-6}$& 307   & 233   & 210   & 231   & 183   & 224   & 193 \\
		&$10^{-8}$ & 1248  & 925   & 421   & 425   & 552   & 697   & 402 \\
		\multirow{3}[0]{*}{P14} & $10^{-4}$ & 18    & 18    & 18    & 21    & 23    & 19    & 20 \\
		& $10^{-6}$ & 77    & 83    & 84    & 78    & 82    & 88    & 69 \\
		&$10^{-8}$ & 885   & 480   & 383   & 338   & 450   & 564   & 274 \\
		\multirow{3}[0]{*}{P15} & $10^{-4}$ & 52    & 50    & 51    & 54    & 54    & 53    & 51 \\
		& $10^{-6}$ & 53    & 53    & 54    & 56    & 56    & 55    & 53 \\
		&$10^{-8}$ & 55    & 53    & 54    & 59    & 295   & 57    & 54 \\
		\multirow{3}[0]{*}{P16} & $10^{-4}$ & 17    & 14    & 14    & 14    & 14    & 16    & 14 \\
		& $10^{-6}$ & 18    & 16    & 17    & 15    & 17    & 19    & 16 \\
		&$10^{-8}$ & 21    & 17    & 17    & 17    & 17    & 21    & 16 \\
		\multirow{3}[0]{*}{P17} & $10^{-4}$ & 14    & 14    & 14    & 14    & 14    & 14    & 14 \\
		& $10^{-6}$ & 90    & 44    & 80    & 37    & 37    & 89    & 33 \\
		& $10^{-8}$ & 95    & 50    & 95    & 41    & 41    & 99    & 38 \\
		\multirow{3}[0]{*}{P18} & $10^{-4}$ & 42    & 30    & 51    & 27    & 40    & 32    & 25 \\
		& $10^{-6}$ & 46    & 38    & 58    & 31    & 52    & 51    & 29 \\
		&$10^{-8}$ & 46    & 44    & 60    & 31    & 57    & 56    & 31 \\
		\multirow{3}[0]{*}{P19} & $10^{-4}$ & 98    & 77    & 107   & 123   & 125   & 105   & 70 \\
		& $10^{-6}$ & 109   & 83    & 114   & 13750 & 321   & 116   & 81 \\
		&$10^{-8}$& 110   & 83    & 117   & 13843 & 398   & 119   & 81 \\
		\multirow{3}[0]{*}{P20} & $10^{-4}$ & 44    & 30    & 44    & 25    & 28    & 40    & 28 \\
		& $10^{-6}$ & 49    & 33    & 51    & 32    & 35    & 43    & 33 \\
		&$10^{-8}$ & 50    & 36    & 51    & 32    & 35    & 46    & 34 \\
		\multirow{3}[0]{*}{P21} & $10^{-4}$ & 84    & 94    & 153   & 346   & 301   & 131   & 94 \\
		& $10^{-6}$ & 89    & 96    & 158   & 484   & 472   & 134   & 100 \\
		& $10^{-8}$ & 90    & 96    & 159   & 486   & 473   & 134   & 100 \\
		\multirow{3}[0]{*}{P22} & $10^{-4}$ & 440   & 343   & 241   & 1541  & 970   & 253   & 218 \\
		& $10^{-6}$& 561   & 436   & 320   & 2991  & 1501  & 371   & 321 \\
		& $10^{-8}$ & 765   & 594   & 393   & 3787  & 2066  & 394   & 390 \\
		\multirow{3}[0]{*}{P23} & $10^{-4}$ & 433   & 409   & 462   & 897   & 887   & 373   & 433 \\
		& $10^{-6}$ & 1227  & 763   & 679   & 2732  & 2743  & 689   & 591 \\
		& $10^{-8}$ & 1669  & 1198  & 726   & 4090  & 5008  & 816   & 630 \\
		\multirow{3}[0]{*}{P24} & $10^{-4}$ & 3     & 16    & 12    & 16    & 16    & 16    & 12 \\
		& $10^{-6}$ & 7     & 22    & 18    & 22    & 25    & 22    & 18 \\
		& $10^{-8}$ & 14    & 29    & 25    & 29    & 35    & 29    & 25 \\
		\multirow{3}[0]{*}[1.2em]{P25} & $10^{-4}$ & 38    & 38    & 40    & 41    & 42    & 44    & 39 \\
		& $10^{-6}$ & 44    & 41    & 45    & 49    & 45    & 53    & 43 \\
		& $10^{-8}$ & 47    & 44    & 49    & 52    & 48    & 57    & 45 \\
		\multirow{3}[0]{*}{P26} & $10^{-4}$ & 47    & 37    & 43    & 88    & 48    & 40    & 38 \\
		& $10^{-6}$ & 801   & 385   & 490   & 629   & 482   & 501   & 368 \\
		& $10^{-8}$ & 6432  & 2457  & 2887  & 4340  & 2652  & 2622  & 2031 \\
		\multirow{3}[0]{*}{P27} & $10^{-4}$ & 25    & 21    & 25    & 24    & 23    & 25    & 22 \\
		& $10^{-6}$ & 33    & 34    & 33    & 33    & 32    & 35    & 33 \\
		& $10^{-8}$ & 45    & 46    & 45    & 49    & 44    & 42    & 41 \\
		\multirow{3}[0]{*}{P28} & $10^{-4}$ & 22    & 22    & 22    & 26    & 22    & 24    & 22 \\
		& $10^{-6}$ & 33    & 29    & 33    & 27    & 26    & 29    & 29 \\
		& $10^{-8}$ & 39    & 34    & 39    & 30    & 29    & 30    & 34 \\
		\multirow{3}[0]{*}{P29} & $10^{-4}$ & 52    & 46    & 41    & 41    & 47    & 41    & 44 \\
		& $10^{-6}$ & 91    & 70    & 75    & 79    & 83    & 80    & 68 \\
		& $10^{-8}$ & 130   & 96    & 132   & 98    & 103   & 125   & 105 \\
		\multirow{3}[0]{*}{P30} & $10^{-4}$ & 7     & 7     & 7     & 7     & 7     & 7     & 7 \\
		& $10^{-6}$ & 12    & 12    & 12    & 18    & 16    & 12    & 12 \\
		& $10^{-8}$ & 56    & 41    & 64    & 59    & 55    & 54    & 49 \\
		\multirow{3}[0]{*}{P31} & $10^{-4}$ & 45    & 51    & 42    & 51    & 47    & 49    & 42 \\
		& $10^{-6}$ & 86    & 75    & 72    & 64    & 73    & 89    & 72 \\
		& $10^{-8}$ & 127   & 99    & 131   & 98    & 100   & 114   & 100 \\
		\multirow{3}[0]{*}{P32} & $10^{-4}$ & 80    & 76    & 78    & 72    & 73    & 64    & 61 \\
		& $10^{-6}$ & 4497  & 431   & 272   & 209   & 301   & 4067  & 253 \\
		&$10^{-8}$ & 11493 & 898   & 478   & 323   & 524   & 9386  & 477 \\
		\multirow{3}[1]{*}{P33} & $10^{-4}$ & 36    & 31    & 34    & 51    & 33    & 40    & 30 \\
		& $10^{-6}$ & 309   & 189   & 279   & 186   & 149   & 200   & 187 \\
		& $10^{-8}$ & 741   & 390   & 613   & 380   & 424   & 481   & 381 \\
		\midrule
		\multirow{3}[2]{*}{Total} & $10^{-4}$ & 3297  & 2368  & 2253  & 4111  & 3549  & 2247  & 1880 \\
		& $10^{-6}$& 12847 & 4953  & 4748  & 122778 & 8629  & 9062  & 3843 \\
		&$10^{-8}$& 34634 & 13479 & 11801 & 138964 & 24044 & 22239 & 9288 \\
\end{longtable}}
\end{center}

\section{Concluding remarks}
In this paper, we propose a parameterized BB (\text{PBB}) method from an interpolated least squares perspective. By introducing interpolation parameter $m_{k}$, the obtained interpolated least squares model is a completeness of the two least squares models in original BB method. The original $\alpha_{k}^{BB1}$, $\alpha_{k}^{BB2}$, and $\sqrt{\alpha_{k}^{BB1}\alpha_{k}^{BB2}}$ correspond to $m_{k}=1$, $0$, and  $\frac{1}{2}$, respectively. For each interpolation parameter $m_{k}$, we interpolate between the classical BB1 and BB2 scalars to obtain a new class of step sizes. We analyze the working principle of the ABB method. Starting from the idea of \text{ABB}, using the variation information between two consecutive iterations, we propose a class of interpolation parameter selection schemes. For a typical two dimensional strictly convex quadratic problem, we analyze the convergence of the PBB method by defining the ratio of the absolute values of the gradient components. Numerical results verify the effectiveness of \text{PBB} method for quadratic and general differentiable objective functions. 

Further exploration of the interpolated least squares model \eqref{VLS} is a meaningful work. Fundamentally, model \eqref{VLS} captures the most essential characteristic of the secant equation: namely, that the degree of the variable $\alpha$ corresponding to $\sss_{k-1}$ is higher by one than that of  $\yy_{k-1}$. This model provides an opportunity to extend the range of the BB1 and BB2 step sizes, i.e., to obtain a scalar that is smaller than $\alpha_{k}^{BB1}$ or larger than $\alpha_{k}^{BB2}$. Recent efforts in this direction have been made in \cite{Xu2025ExtendedVariationalBarzilai}. In addition, we find that the PBB method potentially solves the smoothed problem in \cite{Pandiya2025novelglobaloptimization}. Furthermore, enriching the selection of interpolation parameters $m_{k}$ constitutes another interesting avenue for future work. In addition, as demonstrated in Section \ref{sec:ABB}, the three-term alternating strategy proves to be a more effective approach in the spectral gradient descent method. Therefore, it is also meaningful to develop an efficient three-term alternating step size method.

Last but certainly not least, at the theoretical level, it is worthwhile to investigate the properties of the second-order difference dynamical system corresponding to the spectral gradient method.

\section{Acknowledgments}
The author would like to thank Prof. \text{Congpei An} for his patient guidance.

\bibliographystyle{plain}
\bibliography{VBB}

\end{document}